\documentclass{amsart}

\usepackage{amssymb}
\usepackage{amsmath}
\usepackage{amsthm}
\usepackage{amsfonts}

\usepackage{a4wide}
\usepackage{multirow}
\usepackage{graphicx}
\usepackage{xcolor}

\newtheorem{theorem}{Theorem}[section]

\newtheorem{corollary}{Corollary}[section]

\theoremstyle{definition}
\newtheorem{definition}{Definition}[section]
\newtheorem{remark}{Remark}[section]

\newtheorem{algorithm}{Algorithm}[section]
\newtheorem{conjecture}{Conjecture}[section]

\numberwithin{equation}{section}

\begin{document}


\title
[Modeling rooted in-trees by finite \(p\)-groups]
{Modeling rooted in-trees by finite \(p\)-groups}

\author{Daniel C. Mayer}
\address{Naglergasse 53\\8010 Graz\\Austria}
\email{algebraic.number.theory@algebra.at}
\urladdr{http://www.algebra.at}

\thanks{Research supported by the Austrian Science Fund (FWF): P 26008-N25}


\date{January 27, 2017}

\begin{abstract}
In section
\ref{s:GraphTheory}
we describe the abstract graph theoretic foundations
for a kind of infinite rooted in-trees \(\mathcal{T}(R)=(V,E)\)
with root \(R\), weighted vertices \(v\in V\), and weighted directed edges \(e\in E\subset V\times V\).
The vertex degrees \(\deg(v)\) are always finite
but the trees contain infinite paths \((v_i)_{i\ge 0}\).
In section
\ref{s:GroupTheory}
we introduce a group theoretic model of the rooted in-trees \(\mathcal{T}(R)\).
Vertices are represented by isomorphism classes of finite \(p\)-groups \(G\),
for a fixed prime number \(p\).
Directed edges are represented by epimorphisms \(\pi:G\to\pi{G}\) of finite \(p\)-groups
with characteristic kernels \(\ker(\pi)\).
The weight of a vertex \(G\) is realized by its nuclear rank \(n(G)\)
and the weight of a directed edge \(\pi\) is realized by its step size \(s(\pi)=\log_p(\#\ker(\pi))\).
These invariants are essential for understanding the phenomenon of \textit{multifurcation}.
Since the structure of our rooted in-trees is rather complex,
we use pattern recognition methods for finding finite subgraphs
which repeat indefinitely. Several periodicities admit the reduction
of the complete infinite graph to finite patterns.
Additionally, we employ independent component analysis
for obtaining a graph dissection into pruned subtrees.
As a coronation of this chapter,
we show in section
\ref{s:3ClassTowerLength3}
that \textit{fork topologies} provide a convenient description
of very complex navigation paths through the trees,
arising from repeated multifurcations,
which are of the greatest importance for recent progress in determining \(p\)-class field towers
of algebraic number fields.
\end{abstract}

\maketitle






\section{Underlying abstract graph theory}
\label{s:GraphTheory}
\noindent
Let \(\mathcal{G}=(V,E)\) be a \textit{graph}
with set of \textit{vertices} \(V\)
and set of \textit{edges} \(E\).
We expressly admit infinite sets \(V\) and \(E\),
but we assume that the in- and out-degree of each vertex is finite.



\subsection{Directed edges and paths}
\label{ss:Directed}
\noindent
In this chapter,
we shall be concerned with \textit{directed graphs} (digraphs)
whose edges are rather ordered pairs \((v_1,v_2)\in V\times V\)
than only subsets \(\lbrace v_1,v_2\rbrace\subset V\) with two elements.
Such a \textit{directed edge} \(e=(v_1,v_2)\) is also denoted by an arrow \(v_1\to v_2\)
with starting vertex \(v_1\) and ending vertex \(v_2\).
Thus, we have \(E\subset V\times V\).
Now, infinitude comes in.


\begin{definition}
\label{dfn:Path}
(Finite and infinite paths.) \\
A \textit{finite path} of \textit{length} \(\ell\ge 0\) in \(\mathcal{G}\)
is a finite sequence \((v_i)_{0\le i\le\ell}\) of vertices \(v_i\in V\)
such that \((v_i,v_{i+1})\in E\) for \(0\le i\le\ell-1\).
We call \(v_0\), resp. \(v_\ell\), the starting vertex, resp. ending vertex, of the path.
The degenerate case of a single vertex \((v_0)\)
is called a \textit{point path} of length \(\ell=0\).

An \textit{infinite path} in \(\mathcal{G}\)
is an infinite sequence \((v_i)_{i\ge 0}\) of vertices \(v_i\in V\)
such that \((v_{i+1},v_i)\in E\) for all \(i\ge 0\).
In this case, \(v_0\) is the ending vertex of the path, and there is no starting vertex.
\end{definition}



\subsection{Rooted in-trees with parent operator}
\label{ss:Trees}
\noindent
Our attention will even be restricted to \textit{rooted in-trees} \(\mathcal{G}=\mathcal{T}(R)\),
that is, connected digraphs without cycles
such that the \textit{root vertex} \(R\) has out-degree \(0\)
whereas any other vertex \(v\in V\setminus\lbrace R\rbrace\) has out-degree \(1\).
A vertex with in-degree at least \(1\) is called \textit{capable}
whereas a vertex with in-degree \(0\) is called a \textit{leaf}.
For a rooted in-tree we can define the parent operator as follows.


\begin{definition}
\label{dfn:Parents}
Let \(\mathcal{T}(R)=(V,E)\) be a rooted in-tree.
Then the mapping \(\pi:\,V\setminus\lbrace R\rbrace\to V\), \(v\mapsto\pi{v}\),
where \((v,\pi{v})\in E\) is the unique edge with starting vertex \(v\),
is called the \textit{parent operator} of \(\mathcal{T}(R)\).
For each vertex \(v\in V\),
there exists a unique finite \textit{root path} from \(v\) to the root \(R\),
\[v=\pi^0{v}\to\pi^1{v}\to\pi^2{v}\to\ldots\to\pi^{\ell-1}{v}\to\pi^\ell{v}=R,\]
expressed by iterated applications of the parent operator,
and with some \textit{length} \(\ell\ge 0\).
Each vertex in the root path of \(v\) is called an \textit{ancestor} of \(v\).

The \textit{descendant tree} \(\mathcal{T}(a)=(V(a),E(a))\) of a vertex \(a\in V\)
is the subtree of \(\mathcal{T}(R)=(V,E)\)
consisting of vertices \(v\) with ancestor \(a\),
that is \(v\in V(a):=\lbrace u\in V\mid (\exists\, j\ge 0)\ \pi^j{u}=a\rbrace\),
and edges \(e\in E(a):=E\bigcap (V(a)\times V(a))\).

A vertex \(u\in V\) is called an \textit{immediate descendant} (or \textit{child}) of a vertex \(a\in V\),
if there exists a directed edge \((u,a)\in E\).
In this case, \(a=\pi{u}\) is necessarily the \textit{parent} of \(u\) .
\end{definition}


\noindent
We can define a \textit{partial order} on the vertices \(u,a\in V\) of the tree \(\mathcal{T}(R)\)
by putting \(u\ge a\) if \(u\in\mathcal{T}(a)\), that is,
if \(u\) is descendant of \(a\), and \(a\) is ancestor of \(u\).
The root \(R\) is the minimum.

The root \(R\) is always a common ancestor of two vertices \(u,v\in V\).
By the \textit{fork} of \(u\) and \(v\) we understand their biggest common ancestor,
denoted by \(\mathrm{Fork}(u,v)\),
which admits a measure.


\begin{definition}
\label{dfn:VertexDistance}
(Vertex distance.) \quad
The sum \(\ell_u+\ell_v\) of the path lengths from two vertices \(u,v\in V\) to their fork
is called the \textit{distance} \(d(u,v)\) of the vertices.
\end{definition}



\subsection{Mainlines and multifurcation}
\label{ss:Weights}
\noindent
We shall also need \textit{weight functions}
with non-negative integer values for vertices \(w_V:\,V\to\mathbb{N}_0\),
and with positive integer values for edges \(w_E:\,E\to\mathbb{N}\).
In particular, the sets of vertices and edges have disjoint partitions
\begin{equation}
\label{eqn:WeightPartitions}
\begin{aligned}
V &= \dot{\bigcup}_{n\ge 0}\,V_n \text{ with } V_n:=\lbrace v\in V\mid w_V(v)=n\rbrace \text{ for } n\ge 0, \\
E &= \dot{\bigcup}_{s\ge 1}\,E_s \text{ with } E_s:=\lbrace e\in E\mid w_E(e)=s\rbrace \text{ for } s\ge 1,
\end{aligned}
\end{equation}
such that \(V_0\) is precisely the \textit{set of leaves} of the tree \(\mathcal{T}(R)\).
Thus, there arise weighted measures.


\begin{definition}
\label{dfn:WeightedDistance}
(Path weight and weighted distance.) \\
By the \textit{path weight} of a finite path \((v_i)_{0\le i\le\ell}\)
with length \(\ell\ge 0\) in \(\mathcal{T}(R)\)
such that \((v_i,v_{i+1})\in E_{s_i}\) for \(0\le i\le\ell-1\)
we understand the sum \(\sum_{i=0}^{\ell-1}\,s_i\).
The sum \(w_u+w_v\) of the path weights from two vertices \(u,v\in V\) to their fork
is called the \textit{weighted distance} \(w(u,v)\) of the vertices.
\end{definition}


In Definition
\ref{dfn:Mainline}
and
\ref{dfn:Branches},
some concepts are introduced using the minimal possible weight.

\begin{definition}
\label{dfn:Mainline}
(Mainlines and minimal trees.) \quad
An infinite path \((v_i)_{i\ge 0}\) in \(\mathcal{T}(R)\) with edges of weight \(1\),
that is, such that \((v_{i+1},v_i)\in E_1\) for all \(i\ge 0\),
is called a \textit{mainline} in \(\mathcal{T}(R)\).

The \textit{minimal tree} \(\mathcal{T}_1(a)=(V_1(a),E_1(a))\) of a vertex \(a\in V\)
is the subtree of the descendant tree \(\mathcal{T}(a)=(V(a),E(a))\)
consisting of vertices \(v\),
whose root path in \(\mathcal{T}(a)\) possesses edges \(e\) of weight \(1\) only,
that is \(v\in V_1(a):=\lbrace u\in V(a)\mid (\forall\, 0\le j<\ell)\ (\pi^{j}{u},\pi^{j+1}{u})\in E_1\rbrace\),
and edges \(e\in E_1(a):=E(a)\bigcap (V_1(a)\times V_1(a))\).
\end{definition}


\begin{definition}
\label{dfn:Branches}
(Branches.) \quad
Let \((v_i)_{i\ge 0}\) be a mainline in \(\mathcal{T}(R)\).
For \(i\ge 0\), the difference set
\(\mathcal{B}(v_i):=\mathcal{T}_1(v_i)\setminus\mathcal{T}_1(v_{i+1})\)
of minimal trees
is called the \textit{branch} with root \(v_i\) of the minimal tree \(\mathcal{T}_1(v_0)\).
The branches give rise to a disjoint partition
\(\mathcal{T}_1(v_0)=\dot{\bigcup}_{i\ge 0}\,\mathcal{B}(v_i)\).
\end{definition}


Finally, we complete our abstract graph theoretic language by considering arbitrary weights.

\begin{definition}
\label{dfn:Multifurcation}
(Multifurcation.) \\
Let \(n\ge 2\) be a positive integer.
A vertex \(a\in V_n\) has an \(n\)-fold \textit{multifurcation}
if its in-degree is an \(n\)-fold sum \(N_1+N_2+\ldots+N_n\)
due to \(N_s\ge 1\) incoming edges of weight \(s\), for each \(1\le s\le n\).
That is, we define counters \(N_s\) of all incoming edges of weight \(s\),
and additionally, we have counters \(C_s\) of all incoming edges of weight \(s\) with capable starting vertex,
\begin{equation}
\label{eqn:MultiFurcation}
\begin{aligned}
N_s &:= N_s(a):=\#\lbrace e\in E_s\mid e=(u,a) \text{ for some } u\in V\rbrace, \\
C_s &:= C_s(a):=\#\lbrace e\in E_s\mid e=(u,a) \text{ for some } u\in V \text{ with } w_V(u)\ge 1\rbrace.
\end{aligned}
\end{equation}
We also define an ordering and a notation
\cite{GNO}
for immediate descendants of \(a\) by writing
\(a-\#s;i\) for the \(i\)th immediate descendant with edge of weight \(s\), where \(1\le s\le n\) and \(1\le i\le N_s\).
\end{definition}



\section{Concrete model in \(p\)-group theory}
\label{s:GroupTheory}
\noindent
Now we introduce a group theoretic model of the rooted in-trees \(\mathcal{T}(R)=(V,E)\) in \S\
\ref{s:GraphTheory}.
Vertices \(v\in V\) are represented by isomorphism classes of finite \(p\)-groups \(G\),
for a fixed prime number \(p\).
Directed edges \(e\in E\) are represented by epimorphisms \(\pi:G\to\pi{G}\) of finite \(p\)-groups
with characteristic kernels \(\ker(\pi)=\gamma_c{G}\),
where \(c:=\mathrm{cl}(G)\) denotes the nilpotency class of \(G\)
and \((\gamma_i{G})_{i\ge 1}\) is the lower central series of \(G\).

We emphasize that the symbol \(\pi\) is used now intentionally for two distinct mappings,
the abstract parent operator \(\pi:\,V\setminus\lbrace R\rbrace\to V\), \(v\mapsto\pi{v}\), in Definition
\ref{dfn:Parents},
and the concrete natural projection onto the quotient \(\pi:G\to\pi{G}\simeq G/\gamma_c{G}\), \(g\mapsto\pi(g)=g\cdot\gamma_c{G}\),
for each individual vertex \(G=v\in V\setminus\lbrace R\rbrace\),
which should precisely be denoted by \(\pi=\pi_G\),
but we omit the subscript, since there is no danger of misinterpretation.
In both views, \(\pi{G}\) is the parent of \(G\).

The weight of a vertex \(G\) is realized by its \textit{nuclear rank} \(n(G)\)
\cite[\S\ 14, eqn. (28), p. 178]{Ma6}
and the weight of a directed edge \(\pi:G\to\pi{G}\) is realized by its \textit{step size}
\(s(\pi)=\log_p(\#\gamma_c{G})\)
\cite[\S\ 17, eqn. (33), p. 179]{Ma6}.
These invariants are essential for understanding the phenomenon of \textit{multifurcation} in Definition
\ref{dfn:Multifurcation}.
In particular, we can hide multifurcation by restricting all edges \(\pi\) to step size \(s(\pi)=1\),
that is,
by considering the minimal tree \(\mathcal{T}_1(v)\) instead of the entire descendant tree \(\mathcal{T}(v)\)
of a vertex \(v\in V\).
In our concrete \(p\)-group theoretic model, all vertices \(G\) of a minimal tree share a common \textit{coclass},
which is the additive complement \(\mathrm{cc}(G):=\mathrm{lo}(G)-\mathrm{cl}(G)\)
of the (nilpotency) class \(c=\mathrm{cl}(G)\) with respect to the \textit{logarithmic order}
\(\mathrm{lo}(G):=\log_p(\mathrm{ord}(G))\) of \(G\).
Generally, the logarithmic order of an immediate descendant \(G\) with parent \(\pi{G}\)
increases by the step size, \(\mathrm{lo}(G)=\mathrm{lo}(\pi{G})+s(\pi)\),
since \(\log_p(\#\pi{G})=\log_p(\#(G/\ker{\pi}))=\log_p(\#G)-\log_p(\#\ker{\pi})\).
Consequently, the coclass remains fixed in a minimal tree with \(s(\pi)=1\), since
\[\mathrm{cc}(G)=\mathrm{lo}(G)-\mathrm{cl}(G)=\mathrm{lo}(\pi{G})+1-(\mathrm{cl}(\pi{G})+1)
=\mathrm{lo}(\pi{G})-\mathrm{cl}(\pi{G})=\mathrm{cc}(\pi{G}).\]
A minimal tree \(\mathcal{T}_1(G)\) which contains a unique infinite mainline is called a \textit{coclass tree}.
It is denoted by \(\mathcal{T}^{(r)}(G):=\mathcal{T}_1(G)\)
when its root \(G\) is of coclass \(r:=\mathrm{cc}(G)\).
For further details, see
\cite[\S\ 5, p. 164]{Ma6}.

In view of the principal goals of this chapter,
we must specify our intended situation even more concretely.
We put \(p:=3\), the smallest odd prime number,
and we select as the root 
either \(R:=\langle 243,6\rangle\) or \(R:=\langle 243,8\rangle\),
characterized by its SmallGroup identifier
\cite{BEO2}.
These are metabelian \(3\)-groups of order \(\#R=243=3^5\), logarithmic order \(\mathrm{lo}(R)=5\), class \(c=3\), and coclass \(r=2\).



\subsection{Periodicity of finite patterns}
\label{ss:FirstPeriodicity}
\noindent
Within the frame of the above-mentioned model with \(p=3\)
for the theory of rooted in-trees
as developed in \S\
\ref{s:GraphTheory},
the following finiteness and periodicity statement
becomes provable.


The \textit{virtual periodicity} of depth-pruned branches of coclass trees
has been proven rigorously
with analytic methods (using zeta functions and cone integrals) by du Sautoy
\cite{dS}
in \(2000\),
and with algebraic methods (using cohomology groups) by Eick and Leedham-Green
\cite{EkLg}
in \(2008\).
We recall that a coclass tree contains a unique infinite path
of edges \(\pi\) with uniform step size \(s(\pi)=1\),
the so-called mainline.
Pattern recognition and pattern classification concerns the branches.


\begin{theorem}
\label{thm:FirstPeriodicity}
(A finite periodically repeating pattern.) \\
Among the vertices of any mainline \((v_i)_{i\ge 0}\) in \(\mathcal{T}(R)\),
there exists a \textbf{periodic root} \(v_\varrho\) with \(\varrho\ge 0\)
and a \textbf{period length} \(\lambda\ge 1\) such that the branches
\[\mathcal{B}(v_{i+\lambda})\simeq\mathcal{B}(v_i)\]
are isomorphic \textbf{finite} graphs,
for all \(i\ge\varrho\).
Up to a finite pre-periodic component, the minimal tree \(\mathcal{T}_1(v_0)\)
consists of periodically repeating copies of the finite pattern
\(\dot{\bigcup}_{i=0}^{\lambda-1}\,\mathcal{B}(v_{\varrho+i})\).
\end{theorem}

\begin{proof}
According to
\cite{dS}
and
\cite{EkLg},
the claims are true for pruned branches with any fixed depth.
However, for \(p=3\) and under the pruning operation on \(\mathcal{T}(R)\) described in \S\
\ref{sss:ArtinTransfers},
the virtual periodicity becomes a strict periodicity,
since the depth is bounded uniformly for all branches.
\end{proof}

Before we visualize a particular instance of Theorem
\ref{thm:FirstPeriodicity}
in the diagram of Figure
\ref{fig:GraphDissection},
we have to establish techniques for disentangling dense branches of high complexity.



\subsection{Graph dissection by independent component analysis}
\label{ss:GraphDissection}


\subsubsection{Dissection by Galois action}
\label{sss:GaloisAction}
\noindent
Figure
\ref{fig:GraphDissection}
visualizes a graph dissection
of the tree \(\mathcal{T}(R)\)
by independent component analysis.
This technique drastically reduces the complexity of
visual representations and avoids overlaps of dense subgraphs.
The left hand scale gives the order of groups
whose isomorphism classes are represented by vertices of the graph.
The mainline skeleton (black) connects
branches of non-\(\sigma\) groups (red) in the left subfigure
and branches of \(\sigma\)-groups (green) in the right subfigure.
This terminology has its origin in the action of the Galois group \(\mathrm{Gal}(F/\mathbb{Q})\)
on the abelianization \(\mathfrak{M}/\mathfrak{M}^\prime\),
when a vertex of \(\mathcal{T}(R)\) is realized as
second \(3\)-class group \(\mathfrak{M}:=\mathrm{Gal}\left(F_3^{(2)}/F\right)\)
of an algebraic number field \(F\).
For quadratic fields \(F\), we obtain \(\sigma\)-groups.


\begin{definition}
\label{dfn:SigmaGroup}
A \textit{\(\sigma\)-group} \(G\) admits an automorphism \(\sigma\in\mathrm{Aut}(G)\)
acting as inversion \(\sigma(x)=x^{-1}\) on the commutator quotient \(G/G^\prime\).
\end{definition}
 

The actual graph \(\mathcal{T}(R)\) consists of the overlay (superposition) of both subfigures in Figure
\ref{fig:GraphDissection}.
Infinite mainlines are indicated by arrows.
The periodic bifurcations form an infinite path
with edges of alternating step sizes \(1\) and \(2\),
according to Theorem
\ref{thm:SecondPeriodicity}.
We call it the \textit{maintrunk}.



\begin{figure}[hb]
\caption{Graph dissection into pruned branches connected by the mainline skeleton}
\label{fig:GraphDissection}

\input{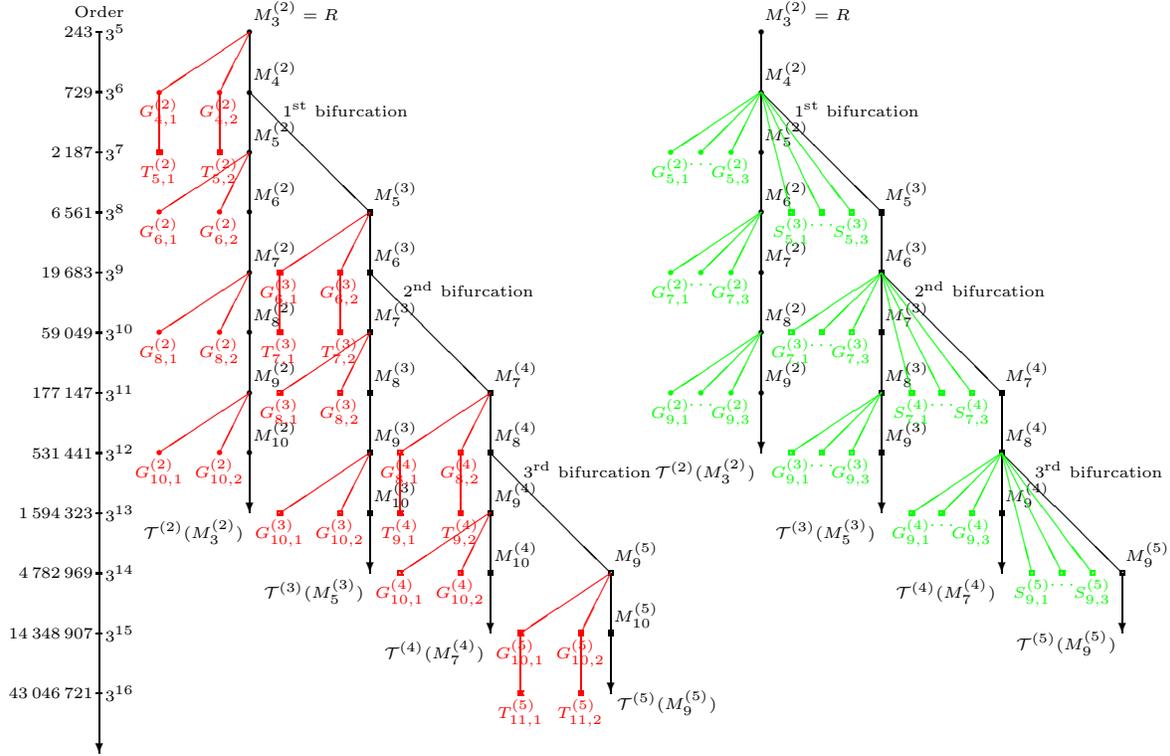}

\end{figure}


With the aid of Figure
\ref{fig:GraphDissection},
a particular instance of Theorem
\ref{thm:FirstPeriodicity}
can be expressed in a more concrete and ostensive way
by taking the tree root as the ending vertex \(v_0:=R\) of the mainline \((v_i)_{i\ge 0}\),
and by using the variable class \(c\ge 3\) and the fixed coclass \(r=2\)
as parameters describing all mainline vertices \(M_c^{(r)}:=v_{c-3}\).
The periodic root is \(M_5^{(2)}=v_{2}\) with \(\varrho=2\) and the period length is \(\lambda=2\).
The finite periodic pattern consists of the two branches
\(\mathcal{B}(M_5^{(2)})=\lbrace M_5^{(2)},G_{6,1}^{(2)},G_{6,2}^{(2)}\rbrace\) (red)
and
\(\mathcal{B}(M_6^{(2)})=\lbrace M_6^{(2)},G_{7,1}^{(2)},G_{7,2}^{(2)},G_{7,3}^{(2)}\rbrace\) (green).
The pre-period is irregular and consists of the two branches
\(\mathcal{B}(M_3^{(2)})=\lbrace M_3^{(2)},G_{4,1}^{(2)},G_{4,2}^{(2)},T_{5,1}^{(2)},T_{5,2}^{(2)}\rbrace\) (red)
and
\(\mathcal{B}(M_4^{(2)})=\lbrace M_4^{(2)},G_{5,1}^{(2)},G_{5,2}^{(2)},G_{5,3}^{(2)}\rbrace\) (green).
But \(M_4^{(2)}\) is not coclass-settled, has nuclear rank \(n=2\), and gives rise to a bifurcation
with immediate descendants \(S_{5,1}^{(3)},S_{5,2}^{(3)},S_{5,3}^{(3)},M_5^{(3)}\) (green) of step size \(s=2\).



\subsubsection{Dissection by Artin transfers}
\label{sss:ArtinTransfers}


{\tiny

\vspace{0.2in}

\begin{figure}[ht]
\caption{Coclass tree \(\mathcal{T}^{(2)}(\langle 243,8\rangle)\) with simple, scaffold and complex types}
\label{fig:TreeUMinDisc}


\setlength{\unitlength}{0.9cm}
\begin{picture}(17,22.5)(-9,-21.5)

\put(-8,0.5){\makebox(0,0)[cb]{Order}}
\put(-8,0){\line(0,-1){20}}
\multiput(-8.1,0)(0,-2){11}{\line(1,0){0.2}}
\put(-8.2,0){\makebox(0,0)[rc]{\(243\)}}
\put(-7.8,0){\makebox(0,0)[lc]{\(3^5\)}}
\put(-8.2,-2){\makebox(0,0)[rc]{\(729\)}}
\put(-7.8,-2){\makebox(0,0)[lc]{\(3^6\)}}
\put(-8.2,-4){\makebox(0,0)[rc]{\(2\,187\)}}
\put(-7.8,-4){\makebox(0,0)[lc]{\(3^7\)}}
\put(-8.2,-6){\makebox(0,0)[rc]{\(6\,561\)}}
\put(-7.8,-6){\makebox(0,0)[lc]{\(3^8\)}}
\put(-8.2,-8){\makebox(0,0)[rc]{\(19\,683\)}}
\put(-7.8,-8){\makebox(0,0)[lc]{\(3^9\)}}
\put(-8.2,-10){\makebox(0,0)[rc]{\(59\,049\)}}
\put(-7.8,-10){\makebox(0,0)[lc]{\(3^{10}\)}}
\put(-8.2,-12){\makebox(0,0)[rc]{\(177\,147\)}}
\put(-7.8,-12){\makebox(0,0)[lc]{\(3^{11}\)}}
\put(-8.2,-14){\makebox(0,0)[rc]{\(531\,441\)}}
\put(-7.8,-14){\makebox(0,0)[lc]{\(3^{12}\)}}
\put(-8.2,-16){\makebox(0,0)[rc]{\(1\,594\,323\)}}
\put(-7.8,-16){\makebox(0,0)[lc]{\(3^{13}\)}}
\put(-8.2,-18){\makebox(0,0)[rc]{\(4\,782\,969\)}}
\put(-7.8,-18){\makebox(0,0)[lc]{\(3^{14}\)}}
\put(-8.2,-20){\makebox(0,0)[rc]{\(14\,348\,907\)}}
\put(-7.8,-20){\makebox(0,0)[lc]{\(3^{15}\)}}
\put(-8,-20){\vector(0,-1){2}}

\put(-6,0.5){\makebox(0,0)[cb]{\(\tau(1)=\)}}
\put(-6,0){\makebox(0,0)[cc]{\((21)\)}}
\put(-6,-2){\makebox(0,0)[cc]{\((2^2)\)}}
{\color{red}
\put(-6,-4){\makebox(0,0)[cc]{\((32)\)}}
}
\put(-6,-6){\makebox(0,0)[cc]{\((3^2)\)}}
{\color{red}
\put(-6,-8){\makebox(0,0)[cc]{\((43)\)}}
}
\put(-6,-10){\makebox(0,0)[cc]{\((4^2)\)}}
{\color{red}
\put(-6,-12){\makebox(0,0)[cc]{\((54)\)}}
}
\put(-6,-14){\makebox(0,0)[cc]{\((5^2)\)}}
\put(-6,-16){\makebox(0,0)[cc]{\((65)\)}}
\put(-6,-18){\makebox(0,0)[cc]{\((6^2)\)}}
\put(-6,-20){\makebox(0,0)[cc]{\((76)\)}}
\put(-6,-21){\makebox(0,0)[cc]{\textbf{TTT}}}
\put(-6.5,-21.2){\framebox(1,22){}}

\put(7.6,-7){\vector(0,1){3}}
\put(7.8,-7){\makebox(0,0)[lc]{depth \(3\)}}
\put(7.6,-7){\vector(0,-1){3}}

\put(-3.1,-8){\vector(0,1){2}}
\put(-3.3,-8){\makebox(0,0)[rc]{period length \(2\)}}
\put(-3.1,-8){\vector(0,-1){2}}

\put(0.7,-2){\makebox(0,0)[lc]{bifurcation from}}
\put(0.7,-2.3){\makebox(0,0)[lc]{\(\mathcal{G}(3,2)\) to \(\mathcal{G}(3,3)\)}}

\multiput(0,0)(0,-2){10}{\circle*{0.2}}
\multiput(0,0)(0,-2){9}{\line(0,-1){2}}
\multiput(-1,-2)(0,-4){5}{\circle*{0.2}}
{\color{red}
\multiput(-1,-4)(0,-4){3}{\circle*{0.2}}
}
\multiput(-1,-16)(0,-4){2}{\circle*{0.2}}
\multiput(-2,-2)(0,-2){10}{\circle*{0.2}}
\multiput(1.95,-4.05)(0,-2){9}{\framebox(0.1,0.1){}}
\multiput(3,-2)(0,-2){10}{\circle*{0.2}}
\multiput(0,0)(0,-2){10}{\line(-1,-2){1}}
\multiput(0,0)(0,-2){10}{\line(-1,-1){2}}
\multiput(0,-2)(0,-2){9}{\line(1,-1){2}}
\multiput(0,0)(0,-2){10}{\line(3,-2){3}}
\multiput(-3.05,-4.05)(-1,0){2}{\framebox(0.1,0.1){}}
\multiput(3.95,-6.05)(0,-2){8}{\framebox(0.1,0.1){}}
\multiput(5,-6)(0,-2){8}{\circle*{0.1}}
\multiput(6,-4)(0,-2){9}{\circle*{0.1}}
\multiput(-1,-2)(-1,0){2}{\line(-1,-1){2}}
\multiput(3,-4)(0,-2){8}{\line(1,-2){1}}
\multiput(3,-4)(0,-2){8}{\line(1,-1){2}}
\multiput(3,-2)(0,-2){9}{\line(3,-2){3}}
\multiput(6.95,-6.05)(0,-2){8}{\framebox(0.1,0.1){}}
\multiput(6,-4)(0,-2){8}{\line(1,-2){1}}

\put(2,-0.5){\makebox(0,0)[lc]{branch}}
\put(2,-0.8){\makebox(0,0)[lc]{\(\mathcal{B}(5)\)}}
\put(2,-2.8){\makebox(0,0)[lc]{\(\mathcal{B}(6)\)}}
\put(2,-4.8){\makebox(0,0)[lc]{\(\mathcal{B}(7)\)}}
\put(2,-6.8){\makebox(0,0)[lc]{\(\mathcal{B}(8)\)}}
\put(2,-8.8){\makebox(0,0)[lc]{\(\mathcal{B}(9)\)}}
\put(2,-10.8){\makebox(0,0)[lc]{\(\mathcal{B}(10)\)}}
\put(2,-12.8){\makebox(0,0)[lc]{\(\mathcal{B}(11)\)}}
\put(2,-14.8){\makebox(0,0)[lc]{\(\mathcal{B}(12)\)}}
\put(2,-16.8){\makebox(0,0)[lc]{\(\mathcal{B}(13)\)}}
\put(2,-18.8){\makebox(0,0)[lc]{\(\mathcal{B}(14)\)}}

\put(-0.1,0.3){\makebox(0,0)[rc]{\(R=\langle 8\rangle\)}}
\put(-2.1,-1.8){\makebox(0,0)[rc]{\(\langle 53\rangle\)}}
\put(-1.1,-1.8){\makebox(0,0)[rc]{\(\langle 55\rangle\)}}
\put(0.1,-1.8){\makebox(0,0)[lc]{\(\langle 54\rangle\)}}
\put(3.1,-1.8){\makebox(0,0)[lc]{\(\langle 52\rangle\)}}
\put(-4.1,-3.5){\makebox(0,0)[cc]{\(\langle 300\rangle\)}}
\put(-3.1,-3.5){\makebox(0,0)[cc]{\(\langle 309\rangle\)}}
\put(-2.1,-3.3){\makebox(0,0)[cc]{\(\langle 302\rangle\)}}
\put(-2.1,-3.5){\makebox(0,0)[cc]{\(\langle 306\rangle\)}}
{\color{red}
\put(-1.1,-3.5){\makebox(0,0)[cc]{\(\langle 304\rangle\)}}
}
\put(0.1,-3.5){\makebox(0,0)[lc]{\(\langle 303\rangle\)}}
\put(2.2,-3.3){\makebox(0,0)[cc]{\(\langle 307\rangle\)}}
\put(2.2,-3.5){\makebox(0,0)[cc]{\(\langle 308\rangle\)}}
\put(3.2,-3.3){\makebox(0,0)[cc]{\(\langle 301\rangle\)}}
\put(3.2,-3.5){\makebox(0,0)[cc]{\(\langle 305\rangle\)}}
\put(6.2,-2.5){\makebox(0,0)[cc]{\(\langle 294\rangle\)}}
\put(6.2,-2.7){\makebox(0,0)[cc]{\(\langle 295\rangle\)}}
\put(6.2,-2.9){\makebox(0,0)[cc]{\(\langle 296\rangle\)}}
\put(6.2,-3.1){\makebox(0,0)[cc]{\(\langle 297\rangle\)}}
\put(6.2,-3.3){\makebox(0,0)[cc]{\(\langle 298\rangle\)}}
\put(6.2,-3.5){\makebox(0,0)[cc]{\(\langle 299\rangle\)}}

\put(-2.1,-5.8){\makebox(0,0)[lc]{\(\langle 2053\rangle\)}}
\put(-1.1,-5.8){\makebox(0,0)[lc]{\(\langle 2051\rangle\)}}
\put(0.1,-5.8){\makebox(0,0)[lc]{\(\langle 2050\rangle\)}}
\put(3.1,-5.8){\makebox(0,0)[lc]{\(\langle 2052\rangle\)}}

{\color{red}
\put(-1.1,-7.8){\makebox(0,0)[lc]{\(\#1;2\)}}
}
\put(0.1,-7.8){\makebox(0,0)[lc]{\(\#1;1\)}}
\put(0.1,-9.8){\makebox(0,0)[lc]{\(\#1;1\)}}
{\color{red}
\put(-1.1,-11.8){\makebox(0,0)[lc]{\(\#1;2\)}}
}

\put(2.1,-3.8){\makebox(0,0)[lc]{\(*2\)}}
\multiput(-2.1,-3.8)(0,-4){5}{\makebox(0,0)[rc]{\(2*\)}}
\multiput(3.1,-3.8)(0,-4){5}{\makebox(0,0)[lc]{\(*2\)}}
\put(6.1,-3.8){\makebox(0,0)[lc]{\(*6\)}}
\multiput(5.1,-5.8)(0,-4){4}{\makebox(0,0)[lc]{\(*2\)}}
\multiput(5.5,-5.3)(0,-4){4}{\makebox(0,0)[lc]{\(\#4\)}}
\multiput(6.1,-5.8)(0,-4){4}{\makebox(0,0)[lc]{\(*2\)}}
\multiput(5.1,-7.8)(0,-4){4}{\makebox(0,0)[lc]{\(*3\)}}
\multiput(6.1,-7.8)(0,-4){4}{\makebox(0,0)[lc]{\(*3\)}}
\multiput(7.1,-7.8)(0,-2){4}{\makebox(0,0)[lc]{\(*2\)}}

\put(-3,-21){\makebox(0,0)[cc]{\textbf{TKT}}}
\put(-2,-21){\makebox(0,0)[cc]{E.9}}
{\color{red}
\put(-1,-21){\makebox(0,0)[cc]{E.8}}
}
\put(0,-21){\makebox(0,0)[cc]{c.21}}
\put(2,-21){\makebox(0,0)[cc]{c.21}}
\put(3.1,-21){\makebox(0,0)[cc]{G.16}}
\put(4,-21){\makebox(0,0)[cc]{G.16}}
\put(5,-21){\makebox(0,0)[cc]{G.16}}
\put(6,-21){\makebox(0,0)[cc]{G.16}}
\put(7,-21){\makebox(0,0)[cc]{G.16}}
\put(-3,-21.5){\makebox(0,0)[cc]{\(\varkappa=\)}}
\put(-2,-21.5){\makebox(0,0)[cc]{\((2231)\)}}
{\color{red}
\put(-1,-21.5){\makebox(0,0)[cc]{\((1231)\)}}
}
\put(0,-21.5){\makebox(0,0)[cc]{\((0231)\)}}
\put(2,-21.5){\makebox(0,0)[cc]{\((0231)\)}}
\put(3.1,-21.5){\makebox(0,0)[cc]{\((4231)\)}}
\put(4,-21.5){\makebox(0,0)[cc]{\((4231)\)}}
\put(5,-21.5){\makebox(0,0)[cc]{\((4231)\)}}
\put(6,-21.5){\makebox(0,0)[cc]{\((4231)\)}}
\put(7,-21.5){\makebox(0,0)[cc]{\((4231)\)}}
\put(-3.8,-21.7){\framebox(11.6,1){}}

\put(0,-18){\vector(0,-1){2}}
\put(0.2,-19.4){\makebox(0,0)[lc]{infinite}}
\put(0.2,-19.9){\makebox(0,0)[lc]{mainline}}
\put(1.8,-20.4){\makebox(0,0)[rc]{\(\mathcal{T}^2(\langle 243,8\rangle)\)}}


\multiput(0,-2)(0,-4){3}{\oval(1,1)}
\put(0.1,-1.3){\makebox(0,0)[lc]{\underbar{\textbf{+540\,365}}}}
\put(0.1,-5.3){\makebox(0,0)[lc]{\underbar{\textbf{+1\,001\,957}}}}
\put(0.1,-9.3){\makebox(0,0)[lc]{\underbar{\textbf{+116\,043\,324}}}}

{\color{red}
\multiput(-1,-4)(0,-4){3}{\oval(1,1)}
\put(-1,-4.8){\makebox(0,0)[lc]{\underbar{\textbf{-34\,867}}}}
\put(-1,-8.8){\makebox(0,0)[lc]{\underbar{\textbf{-370\,740}}}}
\put(-1,-12.8){\makebox(0,0)[lc]{\underbar{\textbf{-4\,087\,295}}}}
}
\put(-1,-16){\oval(1,1)}
\put(-1,-16.8){\makebox(0,0)[lc]{\underbar{\textbf{-19\,027\,947}}}}
\multiput(-2,-4)(0,-4){5}{\oval(1,1)}
\put(-2,-4.8){\makebox(0,0)[rc]{\underbar{\textbf{-9\,748}}}}
\put(-2,-8.8){\makebox(0,0)[rc]{\underbar{\textbf{-297\,079}}}}
\put(-2,-12.8){\makebox(0,0)[rc]{\underbar{\textbf{-1\,088\,808}}}}
\put(-2,-16.8){\makebox(0,0)[rc]{\underbar{\textbf{-11\,091\,140}}}}
\put(-2,-19.3){\makebox(0,0)[rc]{\underbar{\textbf{-94\,880\,548}}}}
\multiput(6,-6)(0,-4){4}{\oval(1,1)}
\put(6,-6.8){\makebox(0,0)[rc]{\underbar{\textbf{-17\,131}}}}
\put(6,-10.8){\makebox(0,0)[rc]{\underbar{\textbf{-819\,743}}}}
\put(6,-14.8){\makebox(0,0)[rc]{\underbar{\textbf{-2\,244\,399}}}}
\put(6,-18.8){\makebox(0,0)[rc]{\underbar{\textbf{-30\,224\,744}}}}

\end{picture}

\end{figure}
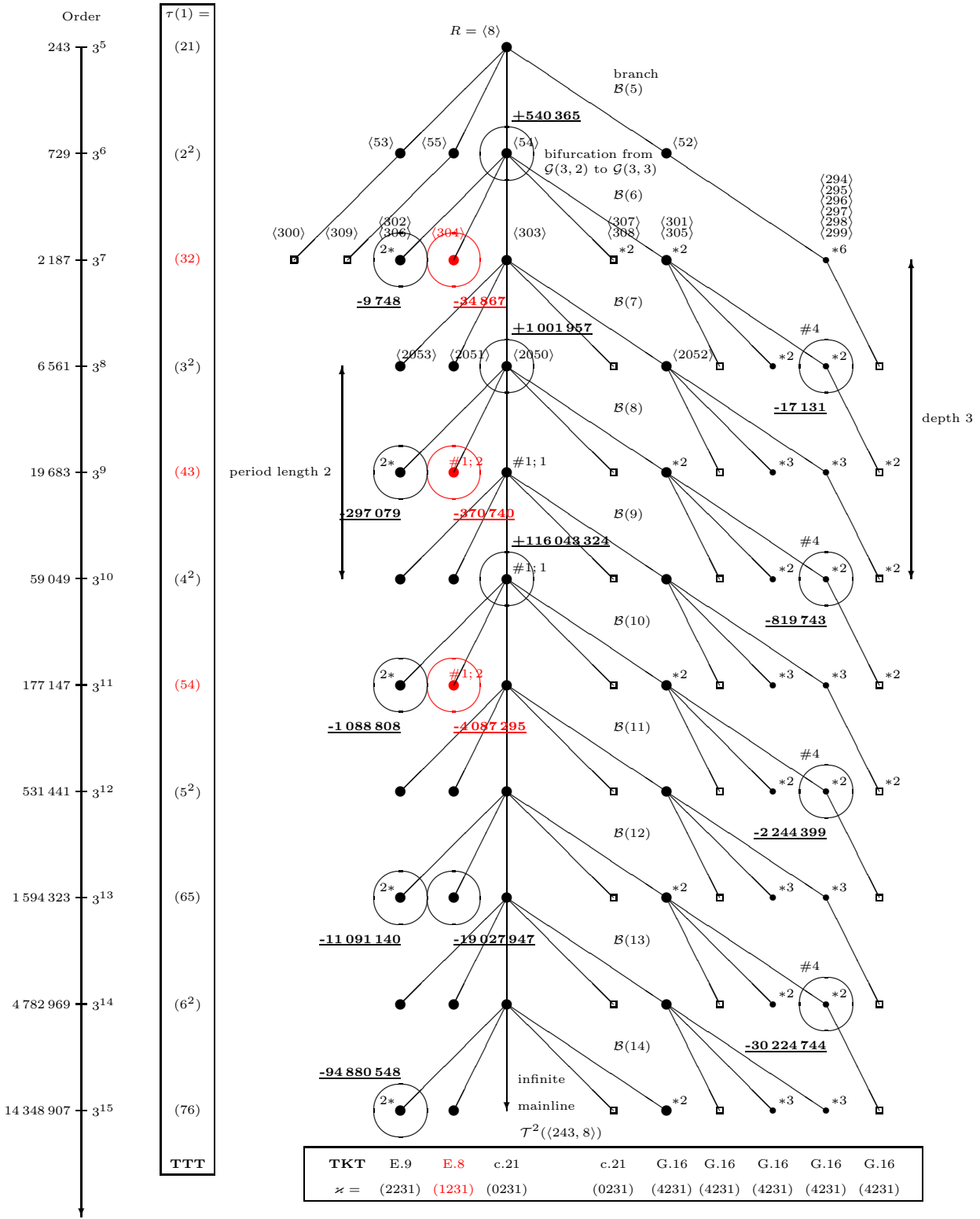

}

In Figure
\ref{fig:GraphDissection},
we have tacitly used a second technique of graph dissection
by independent component analysis. Figure
\ref{fig:TreeUMinDisc}
is restricted to the coclass tree \(\mathcal{T}^{(2)}(R)\) with exemplary root \(R=\langle 243,8\rangle\),
which is the leftmost coclass tree in both subfigures of Figure
\ref{fig:GraphDissection}.
However, now this coclass tree is drawn completely up to logarithmic order \(15\),
containing both, non-\(\sigma\) branches and \(\sigma\)-branches.
The tree is embedded in a kind of coordinate system
having the transfer kernel type (TKT) \(\varkappa\) as its horizontal axis
and the first component \(\tau(1)\) of the transfer target type (TTT) \(\tau\) as its vertical axis
\cite[Dfn. 4.2, p. 27]{Ma10}.
It is convenient to employ a second graph dissection, according to three fundamental types of transfer kernels:
\begin{itemize}
\item
the vertices with \textit{simple} types \(\mathrm{E}.8\), \(\varkappa=(1231)\), and \(\mathrm{E}.9\), \(\varkappa=(2231)\),
which are leaves (left of the mainline), except those of order \(3^6\),
\item
the vertices with \textit{scaffold} (or skeleton) type \(\mathrm{c}.21\), \(\varkappa=(0231)\),
which are either infinitely capable mainline vertices or non-metabelian leaves (immediately right of the mainline),
\item
the vertices with \textit{complex} type \(\mathrm{G}.16\), \(\varkappa=(4231)\),
which are capable at depth \(1\) and give rise to a complicated brushwood of various descendants (right of the mainline).
\end{itemize}
The tacit omission in Figure
\ref{fig:GraphDissection}
concerns all vertices with complex type and the leaves with scaffold type.
Our main results in this chapter will shed complete light on
all mainline vertices and the vertices with simple types.
Underlined boldface integers in Figure
\ref{fig:TreeUMinDisc}
indicate the minimal discriminants \(d\)
of (real and imaginary) quadratic fields \(F=\mathbb{Q}(\sqrt{d})\)
whose second \(3\)-class group \(\mathrm{G}_3^{(2)}{F}:=\mathrm{Gal}\left(F_3^{(2)}/F\right)\)
realizes the vertex surrounded by the adjacent oval.
Three leaves of type \(\mathrm{E}.8\) are drawn with red color,
because they will be referred to in Theorem
\ref{thm:3ClassTowerLength3}
on \(3\)-class towers.



\subsection{Periodicity of infinite patterns}
\label{ss:Periodicity}
\noindent
With the aid of a combination of top down and bottom up techniques,
we are now going to provide evidence of
a new kind of \textit{periodic bifurcations}
in pruned descendant trees
which contain a unique infinite path of edges \(\pi\)
with strictly alternating step sizes \(s(\pi)=1\) and \(s(\pi)=2\),
the so-called \textit{maintrunk}.
It is very important that the trees are \textit{pruned}
in the sense explained at the end of the preceding section \S\
\ref{sss:ArtinTransfers},
for otherwise the maintrunk will not be unique.
In fact, each of our pruned descendant trees \(\mathcal{T}(R)\)
is a countable disjoint union of pruned coclass trees \(\mathcal{T}^{(r)}\), \(r\ge 2\),
which are isomorphic as infinite graphs and connected by edges of weight \(2\),
and finite batches \(\mathcal{T}_0^{(r)}\), \(r\ge 3\), of sporadic vertices outside of coclass trees. 
The top down and bottom up techniques are implemented simultaneously
in two recursive Algorithms
\ref{alg:Mainlines}
and
\ref{alg:ShafarevichCover}.

The first Algorithm
\ref{alg:Mainlines}
recursively constructs the mainline vertices \(M_{c}^{(r)}\), with class \(c\ge 2r-1\),
of the coclass tree \(\mathcal{T}^{(r)}\subset\mathcal{T}(R)\),
for an assigned value \(r\ge 2\),
by means of the \textit{bottom up} technique.
In each recursion step,
the \textit{top down} technique is used for constructing the class-\(c\) quotient \(\mathcal{L}_c^{(r)}\)
of an infinite limit group \(\mathcal{L}^{(r)}\).
Finally, the isomorphism \(M_{c}^{(r)}\simeq\mathcal{L}_c^{(r)}\) is proved.


\begin{theorem}
\label{thm:SecondPeriodicity}
(An infinite periodically repeating pattern.) \quad Let \(u_r=30\) be an upper bound.
An infinite path is generated recursively, since for each \(2\le r<u_r\),
the immediate descendant \(M_{2r+1}^{(r+1)}=M_{2r}^{(r)}-\#2;1\) of step size \(2\)
of the second mainline vertex \(M_{2r}^{(r)}\) of the coclass tree \(\mathcal{T}^{(r)}(M_{2r-1}^{(r)})\),
is root of a new coclass tree \(\mathcal{T}^{(r+1)}(M_{2(r+1)-1}^{(r+1)})\).
The pruned coclass trees
\[\mathcal{T}^{(r)}(M_{2r-1}^{(r)})\simeq\mathcal{T}^{(2)}(M_3^{(2)})\]
are isomorphic \textbf{infinite} graphs,
for each \(2\le r\le u_r\). Note that the nuclear rank \(n(M_{2r}^{(r)})=2\).
\end{theorem}

\noindent
This is the first main theorem of the present chapter.
The proof will be conducted with the aid of an infinite limit group \(\mathcal{L}_{\pm}\),
due to M. F. Newman.
Certain quotients of \(\mathcal{L}_{\pm}\) give precisely the mainline vertices \(M_c^{(r)}\)
with \(r\ge 2\) and \(c\ge 2r-1\),
as will be shown in Theorem
\ref{thm:Mainlines},
Remark
\ref{rmk:MainlineConjecture}.

\begin{conjecture}
\label{cnj:SecondPeriodicity}
Theorem
\ref{thm:SecondPeriodicity}
remains true for any upper bound \(u_r>30\).
\end{conjecture}



\subsection{Mainlines of the pruned descendant tree \(\mathcal{T}(R)\)}
\label{ss:Mainlines}

\begin{definition}
\label{dfn:Mainlines}
The complete theory of the mainlines in \(\mathcal{T}(R)\) is based on the group
\begin{equation}
\label{eqn:ProfiniteGroup}
\mathcal{L}_{\pm}:=\langle\ a,t\ \mid\ (at)^3=a^3,\ \lbrack\lbrack t,a\rbrack,t\rbrack=a^{\pm 3}\ \rangle.
\end{equation}
For each \(r\ge 2\),
quotients of \(\mathcal{L}_{\pm}\) are defined by
\begin{equation}
\label{eqn:MainlineLimits}
\mathcal{L}_{\pm}^{(r)}:=\mathcal{L}_{\pm}\ /\ \langle\ a^{3^r}\ \rangle.
\end{equation}
For each \(r\ge 2\), and for each \(c\ge 2r-1\),
quotients of \(\mathcal{L}_{\pm}^{(r)}\) are defined by
\begin{equation}
\label{eqn:MainlineVertices}
\mathcal{L}^{(r)}_{\pm,c}:=
\begin{cases}
\mathcal{L}_{\pm}^{(r)}\ /\ \langle\ \lbrack t,a\rbrack^{3^\ell}\ \rangle & \text{ if } c=2\ell+1 \text{ odd },\ \ell\ge r-1, \\
\mathcal{L}_{\pm}^{(r)}\ /\ \langle\ t^{3^\ell}\ \rangle                  & \text{ if } c=2\ell \text{ even },\ \ell\ge r.
\end{cases}
\end{equation}
\end{definition}


The following Algorithm
\ref{alg:Mainlines}
is based on iterated applications of the \(p\)-group generation algorithm by Newman
\cite{Nm}
and O'Brien
\cite{OB}.
It starts with the root \(R\), given by its compact presentation,
and constructs an initial section of the unique infinite maintrunk
with strictly alternating step sizes \(1\) and \(2\)
in the pruned descendant tree \(\mathcal{T}(R)\).
In each step, the required selection of the child with appropriate transfer kernel type (TKT)
is achieved with the aid of our own subroutine \texttt{IsAdmissible()},
which is an elaborate version of
\cite[\S\ 4.1, p. 76]{Ma9}.
After reaching an assigned coclass \texttt{r\(=\)hb\(+2\)},
our algorithm navigates along the mainline of the coclass tree \(\mathcal{T}^{(r)}\subset\mathcal{T}(R)\)
and tests each vertex for isomorphism to the corresponding quotient \(\mathcal{L}^{(r)}_{\pm,c}\)
of class \texttt{c\(\le\)2r\(-\)1\(+\)vb}.


\begin{algorithm}
\label{alg:Mainlines}
(Mainline vertices.) \\
\textbf{Input:}
prime \texttt{p}, compact presentation \texttt{cp} of the root, bounds \texttt{hb,vb}, sign \texttt{s}. \\
\textbf{Code:}
uses the subroutine \texttt{IsAdmissible()}.
\texttt{
\begin{tabbing}
   for \= for \= for \= \kill
   r := 2; // initial coclass\\
   Root := PCGroup(cp);\\
   for i in [1..hb] do // bottom up in double steps along the maintrunk\+\\
      Des := Descendants(Root,NilpotencyClass(Root)+1: StepSizes:=[1]);\\
      for j in [1..\(\#\)Des] do\+\\
         if IsAdmissible(Des[j],p,0) then\+\\
            Root := Des[j];\-\\
         end if;\-\\
      end for;\\
      r := r + 1; // coclass recursion\\
      Des := Descendants(Root,NilpotencyClass(Root)+1: StepSizes:=[2]);\\
      for j in [1..\(\#\)Des] do\+\\
         if IsAdmissible(Des[j],p,0) then\+\\
            Root := Des[j];\-\\
         end if;\-\\
      end for;\-\\
   end for;\\
   c := 2\(\ast\)r - 1; // starting class c in dependence on the coclass r\\
   er := p\({}\,\hat{}\,{}\)r; l := (c - 1) div 2; ec := p\({}\,\hat{}\,{}\)l;\\
   M<a,t> := Group<a,t|(a\(\ast\)t)\({}\,\hat{}\,{}\)p=a\({}\,\hat{}\,{}\)p,((t,a),t)=a\({}\,\hat{}\,{}\)(s\(\ast\)p),a\({}\,\hat{}\,{}\)er=1,(t,a)\({}\,\hat{}\,{}\)ec=1>;\\
   QM,pM := pQuotient(M,p,c); // top down construction\\
   if IsIsomorphic(Root,QM) then // identification\+\\
      printf "Isomorphism for cc=\(\%\)o, cl=\(\%\)o.\(\backslash\)n",r,c;\-\\
   end if;\\
   for i in [1..vb] do // bottom up in single steps along a mainline\+\\
      c := c + 1; // nilpotency class recursion\\
      if (0 eq c mod 2) then // even nilpotency class\+\\
         l := c div 2; ec := p\({}\,\hat{}\,{}\)l;\\
         M<a,t> := Group<a,t|(a\(\ast\)t)\({}\,\hat{}\,{}\)p=a\({}\,\hat{}\,{}\)p,((t,a),t)=a\({}\,\hat{}\,{}\)(s\(\ast\)p),a\({}\,\hat{}\,{}\)er=1,t\({}\,\hat{}\,{}\)ec=1>;\-\\
      else // odd nilpotency class\+\\
         l := (c - 1) div 2; ec := p\({}\,\hat{}\,{}\)l;\\
         M<a,t> := Group<a,t|(a\(\ast\)t)\({}\,\hat{}\,{}\)p=a\({}\,\hat{}\,{}\)p,((t,a),t)=a\({}\,\hat{}\,{}\)(s\(\ast\)p),a\({}\,\hat{}\,{}\)er=1,(t,a)\({}\,\hat{}\,{}\)ec=1>;\-\\
      end if;\\
      QM,pM := pQuotient(M,p,c); // top down construction\\
      Des := Descendants(Root,NilpotencyClass(Root)+1: StepSizes:=[1]);\\
      for j in [1..\(\#\)Des] do\+\\
         if IsAdmissible(Des[j],p,0) then\+\\
            Root := Des[j];\-\\
         end if;\-\\
      end for;\\
      if IsIsomorphic(Root,QM) then // identification\+\\
         printf "Isomorphism for cc=\(\%\)o, cl=\(\%\)o.\(\backslash\)n",r,c;\-\\
      end if;\-\\
   end for;
\end{tabbing}
}
\noindent
\textbf{Output:}
coclass \texttt{r} and class \texttt{c} in each case of an isomorphism.
\end{algorithm}


\begin{remark}
\label{rmk:Mainlines}
Algorithm
\ref{alg:Mainlines}
is designed to be called with input parameters the prime
\texttt{p\(=\)3}
and \texttt{cp} the compact presentation of either
the root \(\langle 243,6\rangle\) with sign \texttt{s\(=\)-1}
or
the root \(\langle 243,8\rangle\) with sign \texttt{s\(=\)+1}.
In the current version V2.22-7 of the computational algebra system MAGMA
\cite{MAGMA},
the bounds are restricted to \texttt{r\(=\)hb\(+2\)\(\le\)8} and \texttt{c\(=\)vb\(+\)2r\(-\)1\(\le\)35},
since otherwise the maximal possible internal word length of relators in MAGMA is surpassed.
Close to these limits, the required random access memory increases to a considerable value of approximately \(8\,\)GB RAM.
\end{remark}


\begin{theorem}
\label{thm:Mainlines}
(Mainline vertices as quotients of the limit group \(\mathcal{L}_{\pm}\).) Let \(u_r:=8\), \(u_c:=35\).
\begin{enumerate}
\item
For each \(2\le r\le u_r\), and for each \(2r-1\le c\le u_c\),
the mainline vertex \(M_c^{(r)}\) of coclass \(r\) and nilpotency class \(c\) in the tree \(\mathcal{T}(R)\) is isomorphic to \(\mathcal{L}^{(r)}_{\pm,c}\).
\item
For each \(2\le r\le u_r\),
the projective limit of the mainline \(\left(M_c^{(r)}\right)_{c\ge 2r-1}\) with vertices of coclass \(r\) in the tree \(\mathcal{T}(R)\) is isomorphic to \(\mathcal{L}_{\pm}^{(r)}\).
\item
\(\mathcal{L}_{\pm}\) is an infinite non-nilpotent profinite limit group.
\end{enumerate}
\end{theorem}

\begin{proof}
(1) The repeated execution of Algorithm
\ref{alg:Mainlines}
for successive values from \texttt{hb:=0} to \texttt{hb:=6},
with input data
\texttt{p:=3}, \texttt{cp:=CompactPresentation(SmallGroup(243,i))}, \(i\in\lbrace 6,8\rbrace\),
\(s\in\lbrace -1,+1\rbrace\), and \texttt{vb:=32},
proves the isomorphisms \(M_c^{(r)}\simeq\mathcal{L}^{(r)}_{\pm,c}\) for \(2\le r\le u_r=8\) and \(2r-1\le c\le u_c=35\).
The algorithm is initialized by the starting group \(R=M_3^{(2)}=\langle 243,i\rangle\) of coclass \texttt{r:=2}.
The first loop moves along the maintrunk recursively with strictly alternating step sizes \(1\) and \(2\)
until the root \(M_{2r-1}^{(r)}\) of the coclass tree \(\mathcal{T}^{(r)}\) with \texttt{r=2+hb} is reached.
The second loop iterates through the mainline vertices \(M_{c}^{(r)}\), \(c\ge 2r-1\), of the coclass tree
\(\mathcal{T}^{(r)}\left(M_{2r-1}^{(r)}\right)\),
always checking for isomorphism to the appropriate quotient \(\mathcal{L}^{(r)}_{\pm,c}\).
The subroutine \texttt{IsAdmissible()} tests the transfer kernel type of all descendants
and selects the unique capable descendant with type \(\mathrm{c}.18\) resp. \(\mathrm{c}.21\). \\
(2) Since periodicity sets in for \(2u_r-1=17\le c\le u_c=35\),
the claim is a consequence of Theorem
\ref{thm:FirstPeriodicity}.
(3) The quotient \(\mathcal{L}_{\pm}^{(1)}\) is already infinite and non-nilpotent.
Adding the relation \(\lbrack t,t^a,t\rbrack=1\) suffices to give
\(\lbrack t,a,t\rbrack\) central and \(\mathcal{L}_{\pm}\) profinite.
\end{proof}


\begin{conjecture}
\label{cnj:Mainlines}
Theorem
\ref{thm:Mainlines}
remains true for arbitrary upper bounds \(u_r>8\), \(u_c>35\).
\end{conjecture}


\begin{remark}
\label{rmk:MainlineConjecture}
When the top down constructions in Algorithm
\ref{alg:Mainlines}
are cancelled, the bottom up operations are still able to
establish much bigger initial sections of the infinite maintrunk and of the infinite coclass tree
with fixed coclass \(r\ge 2\).
Admitting an increasing amount of CPU time, we can easily reach astronomic values
of the coclass, \(r=32\), and the nilpotency class, \(c=63\),
that is a logarithmic order of \(r+c=95\),
without surpassing any internal limitations of MAGMA,
and the required storage capacity remains quite modest, i.e., clearly below \(1\,\)GB RAM.
This remarkable stability underpins Conjecture
\ref{cnj:Mainlines}
with additional support from the bottom up point of view.
\end{remark}



\subsection{Covers of metabelian \(3\)-groups}
\label{ss:CoverMetab}
\noindent
Only one of the coclass subtrees \(\mathcal{T}^{(r)}\), \(r\ge 2\),
of the entire rooted in-tree \(\mathcal{T}(R)\)
contains metabelian vertices, namely the first subtree \(\mathcal{T}^{(2)}\).
The following theorem shows how transfer kernel types
are distributed among metabelian vertices \(G\) of depth \(\mathrm{dp}(G)\le 1\) on the tree \(\mathcal{T}^{(2)}\),
as partially illustrated by the Figures
\ref{fig:GraphDissection}
and
\ref{fig:TreeUMinDisc}. 

\begin{theorem}
\label{thm:Metab}
(Metabelian vertices of the coclass tree \(\mathcal{T}^{(2)}{R}\).) \\
For each finite \(3\)-group \(G\),
we denote by \(c:=\mathrm{cl}(G)\) the nilpotency class,
by \(r:=\mathrm{cc}(G)\) the coclass,
and by \(\varkappa\) the transfer kernel type of \(G\).
More explicitly, such a group is also denoted by \(G=G_c^{(r)}\).
The following statements describe the structure of the metabelian skeleton
of the coclass tree \(\mathcal{T}^{(2)}{R}\) with root \(R:=\langle 243,6\rangle\), resp. \(R:=\langle 243,8\rangle\),
down to depth \(1\).
\begin{enumerate}
\item
For each \(c\ge 3\), the mainline vertex \(M_{c}^{(2)}\) of the coclass tree
possesses type \(\mathrm{c}.18\), \(\varkappa=(0122)\), resp. \(\mathrm{c}.21\), \(\varkappa=(0231)\).
\item
For each \(c\ge 4\),
there exists a unique child \(G_{c,1}^{(2)}\) of \(M_{c-1}^{(2)}\)
with type \(\mathrm{E}.6\), \(\varkappa=(1122)\), resp. \(\mathrm{E}.8\), \(\varkappa=(1231)\).
\item
For even \(c\ge 4\),
there exists a unique child \(G_{c,2}^{(2)}\) of \(M_{c-1}^{(2)}\)
with type \(\mathrm{E}.14\), \(\varkappa=(3122)\), resp. \(\mathrm{E}.9\), \(\varkappa=(2231)\).
Thus, \(N_1(M_{c-1}^{(2)})=3\) and \(C_1(M_{c-1}^{(2)})=1\), in the pruned tree.
\item
For odd \(c\ge 5\),
there exist two children \(G_{c,2}^{(2)}\), \(G_{c,3}^{(2)}\) of \(M_{c-1}^{(2)}\)
with type \(\mathrm{E}.14\), \(\varkappa=(3122)\sim (4122)\), resp. \(\mathrm{E}.9\), \(\varkappa=(2231)\sim (3231)\).
Thus, \(N_1(M_{c-1}^{(2)})=4\) and \(C_1(M_{c-1}^{(2)})=1\).
\item
For even \(c\ge 4\),
there exists a unique child \(G_{c,4}^{(2)}\) of \(M_{c-1}^{(2)}\)
with type \(\mathrm{H}.4\), \(\varkappa=(2122)\), resp. \(\mathrm{G}.16\), \(\varkappa=(4231)\).
It is removed from the pruned tree.
\item
For odd \(c\ge 5\),
there exist two children \(G_{c,4}^{(2)}\), \(G_{c,5}^{(2)}\) of \(M_{c-1}^{(2)}\)
with type \(\mathrm{H}.4\), \(\varkappa=(2122)\), resp. \(\mathrm{G}.16\), \(\varkappa=(4231)\).
They are removed from the pruned tree.
\end{enumerate}
\end{theorem}

\begin{proof}
See Nebelung
\cite[Lemma 5.2.6, p. 183, Figures, p. 189 f., Satz 6.14, p. 208]{Ne}.
\end{proof}


\begin{definition}
\label{dfn:CoverLimit}
For \(e\in\lbrace 0,1\rbrace\), we define the \textit{cover limit}, due to M. F. Newman, to be the group
\begin{equation}
\label{eqn:CoverLimit}
\begin{aligned}
\mathcal{C}^{(e)}
:= \langle&\ a,t,u,y,z\ \mid\ t^a=u,\ u^atuy=\lbrack u,t\rbrack^e,\ a^3\lbrack t,a,t\rbrack=z,\ \lbrack u,t,t\rbrack=\lbrack u,t,u\rbrack=1, \\
          &\ y^3=1,\ \lbrack a,y\rbrack=\lbrack t,y\rbrack=\lbrack u,y\rbrack=\lbrack z,y\rbrack=1,\ z^3 =1,\ \lbrack t,z\rbrack=\lbrack u,z\rbrack=1\ \rangle.
\end{aligned}
\end{equation}
See
\cite{Ma5s}.
For each \(k\in\lbrace -1,0,1\rbrace\) and for each integer \(c\ge 4\), let
\begin{equation}
\label{eqn:Quotient}
\mathcal{Q}_c^{(e,k)}:=\mathcal{C}^{(e)}\ /\ \langle\ yw_c^kv_c,\ zw_c\ \rangle
\end{equation}
be the \textit{class-\(c\) quotient with parameter} \(k\) of \(\mathcal{C}^{(e)}\), where
\(w_c:=\lbrack t,\overbrace{a,\ldots,a}^{(c-1) \text{ times}}\rbrack\) and
\(v_c:=\lbrack w_{c-2},\lbrack t,a\rbrack\rbrack\).
\end{definition}


In each step, \(i\ge 1\), of the second Algorithm
\ref{alg:ShafarevichCover},
the \textit{top down} technique constructs a certain class-\(c\) quotient \(\mathcal{Q}_c\), \(c=i+3\),
of a fixed infinite pro-\(3\) group \(\mathcal{C}\), the \textit{cover limit},
and the \textit{bottom up} technique constructs all metabelian children
of a certain vertex \(M_{i-1}\) on the mainline of the first coclass tree \(\mathcal{T}^{(2)}(R)\),
and selects,
firstly, the next vertex \(M_i\) of depth \(\mathrm{dp}(M_i)=0\) on the mainline of \(\mathcal{T}^{(2)}(R)\)
for continuing the recursion,
secondly, a vertex \(G_i\) of depth \(\mathrm{dp}(G_i)=1\) with assigned transfer kernel type \(\varkappa(G_i)\).
Each recursion step is completed by proving that
\(G_i\) is isomorphic to the second derived quotient \(\mathcal{Q}_c/\mathcal{Q}_c^{\prime\prime}\),
that is, \(\mathcal{Q}_c\in\mathrm{cov}(G_i)\) belongs to the \textit{cover} of \(G_i\) in the sense of
\cite[\S\ 1.3, Dfn. 1.1, p. 75]{Ma15}.
More precisely, we have \(M_i=M_{i+3}^{(2)}\) and \(G_i=G_{i+3,j}^{(2)}\) with some \(j\).

\begin{algorithm}
\label{alg:ShafarevichCover}
(Shafarevich cover.) \\
\textbf{Input:}
prime \texttt{p}, compact presentation \texttt{cp} of the root, bound \texttt{vb}, parameters \texttt{e} and \texttt{k}. \\
\textbf{Code:}
uses the subroutine \texttt{IsAdmissible()}.
\texttt{
\begin{tabbing}
for \= for \= for \= for \= \kill
C<a,t,u,y,z> := Group< a,t,u,y,z |\+\\
                y\({}\,\hat{}\,{}\)p, (a,y), (t,y), (u,y), (y,z), (t,z), (u,z), z\({}\,\hat{}\,{}\)p,\\
                (u,t,t), (u,t,u), t\({}\,\hat{}\,{}\)a = u, u\({}\,\hat{}\,{}\)a\(\ast\)t\(\ast\)u\(\ast\)y\(\ast\)(u,t)\({}\,\hat{}\,{}\)-e, a\({}\,\hat{}\,{}\)p\(\ast\)(t,a,t) = z >;\-\\
Root := PCGroup(cp);\\
Leaf := Root;\\
for i in [1..vb] do // bottom up along the mainline of coclass 2\+\\ 
   c := i + 3; // nilpotency class\\
   w := [t];\\
   for j in [1..c] do // construction of iterated commutator\+\\
      s := (w[j],a);\\
      Append(\(\sim\)w,s);\-\\
   end for;\\
   w1 := w[c-2]\({}\,\hat{}\,{}\)-1\(\ast\)(a,t)\(\ast\)w[c-2]\(\ast\)(t,a);\\
   H := quo<C | y\(\ast\)w[c]\({}\,\hat{}\,{}\)k\(\ast\)w1, z\(\ast\)w[c]>;\\
   Q,pQ := pQuotient(H,p,c); // top down construction of Shafarevich cover\\  
   Des := Descendants(Root,NilpotencyClass(Root)+1);\\
   m := 0;\\
   for cnt in [1..\(\#\)Des] do\+\\
      if IsAdmissible(Des[cnt],p,0) then\+\\
         Root := Des[cnt]; // next mainline vertex\-\\
      elif IsAdmissible(Des[cnt],p,2) then\+\\
         m := m + 1;\\
         if (1 eq m) then\+\\
            Leaf := Des[cnt]; // first leaf with assigned TKT\-\\
         end if;\-\\
      end if;\-\\
   end for;\\
   DQ := DerivedSubgroup(Q);\\
   D2Q := DerivedSubgroup(DQ);\\
   Q2Q := Q/D2Q; // metabelianization\\
   if IsIsomorphic(Leaf,Q2Q) then // identification\+\\
      printf "Dsc.cl.\(\%\)o isomorphic to 2nd drv.qtn.\(\,\)of Cov.cl.\(\%\)o.\(\backslash\)n",c,c;\-\\
   end if;\-\\
end for;
\end{tabbing}
}
\noindent
\textbf{Output:}
nilpotency class \texttt{c} in each case of an isomorphism.
\end{algorithm}


The next theorem is the second main result of this chapter,
establishing the finiteness and structure of the \textit{cover}
for each metabelian \(3\)-group with transfer kernel type in section \(\mathrm{E}\).

\begin{theorem}
\label{thm:ExplicitCovers}
(Explicit covers of metabelian \(3\)-groups.)
Let \(u:=8\) be an upper bound
and \(G_{c,j}^{(2)}\) in \(\mathcal{T}^{(2)}\left(M_{3}^{(2)}\right)\) be the metabelian \(3\)-group of nilpotency class \(c\ge 4\)
with transfer kernel type
\[\varkappa=
\begin{cases}
(1122),\ \mathrm{E}.6, \text{ resp. } (1231),\ \mathrm{E}.8 & \text{ if } j=1, \\
(3122),\ \mathrm{E}.14, \text{ resp. } (2231),\ \mathrm{E}.9 & \text{ if } j=2 \text{ or } (j=3 \text{ and } c \text{ odd}).
\end{cases}\]
\begin{enumerate}
\item
The \textbf{cover} of \(G_{c,j}^{(2)}\) is given by
\begin{equation}
\label{eqn:CovSecE}
\mathrm{cov}\left(G_{c,j}^{(2)}\right)=
\begin{cases}
\left\lbrace G_{c,j}^{(2)};G_{c,j}^{(3)},\ldots,G_{c,j}^{(\ell+1)},G_{c,j}^{(\ell+2)},T_{c+1,j}^{(\ell+2)}\right\rbrace & \text{ if } c=2\ell+4,\ 1\le j\le 2, \\
\left\lbrace G_{c,j}^{(2)};G_{c,j}^{(3)},\ldots,G_{c,j}^{(\ell+1)},G_{c,j}^{(\ell+2)},S_{c,j}^{(\ell+3)}\right\rbrace   & \text{ if } c=2\ell+5,\ 1\le j\le 3.
\end{cases}
\end{equation}
where \(0\le\ell\le u\).
In particular, the cover is a finite set with \(\ell+2\) elements
(\(\ell+1\) of them non-trivial),
which are non-\(\sigma\) groups for even \(c\ge 4\), and \(\sigma\)-groups for odd \(c\ge 5\).
\item
The \textbf{Shafarevich cover} of \(G_{c,j}^{(2)}\) with respect to imaginary quadratic fields \(F\) is given by
\begin{equation}
\label{eqn:ShafarevichCovSecE}
\mathrm{cov}\left(G_{c,j}^{(2)},F\right)=
\begin{cases}
\emptyset                                    & \text{ if } c=2\ell+4,\ 0\le\ell\le u,\ 1\le j\le 2, \\
\left\lbrace S_{c,j}^{(\ell+3)}\right\rbrace & \text{ if } c=2\ell+5,\ 0\le\ell\le u,\ 1\le j\le 3.
\end{cases}
\end{equation}
In particular, the Shafarevich cover contains a unique Schur \(\sigma\)-group, if \(c\ge 5\) is odd.
\item
The class-\(c\) quotient with parameter \(k\) of the cover limit \(\mathcal{C}^{(e)}\) is isomorphic to a Schur \(\sigma\)-group
\(\mathcal{Q}^{(e,k)}_c\simeq S_{c,j}^{(\ell+3)}\), for \(c=2\ell+5\),
or to a non-\(\sigma\) group
\(\mathcal{Q}^{(e,k)}_c\simeq G_{c,j}^{(\ell+2)}\), for \(c=2\ell+4\).
The precise correspondence between the parameters \(k\) and \(j\) is given in the following way.
\begin{equation}
\label{eqn:IsomCovSecE}
\begin{aligned}
\text{Types }\mathrm{E}.6,\mathrm{E}.8:\ \mathcal{Q}^{(e,0)}_c &\simeq
\begin{cases}
S_{c,1}^{(\ell+3)} & \text{ for odd class } c=2\ell+5,\ 0\le\ell\le u, \\
G_{c,1}^{(\ell+2)} & \text{ for even class } c=2\ell+4,\ 0\le\ell\le u, \\
\end{cases}
\\
\text{type }\mathrm{E}.9:\ \mathcal{Q}^{(+1,-1)}_c &\simeq
\begin{cases}
S_{c,2}^{(\ell+3)} & \text{ for odd class } c=2\ell+5,\ 0\le\ell\le u, \\
G_{c,2}^{(\ell+2)} & \text{ for even class } c=2\ell+4,\ 0\le\ell\le u, \\
\end{cases}
\\
\text{type }\mathrm{E}.9:\ \mathcal{Q}^{(+1,+1)}_c &\simeq
\begin{cases}
S_{c,3}^{(\ell+3)} & \text{ for odd class } c=2\ell+5,\ 0\le\ell\le u, \\
G_{c,2}^{(\ell+2)} & \text{ for even class } c=2\ell+4,\ 0\le\ell\le u. \\
\end{cases}
\end{aligned}
\end{equation}
In particular, \(\mathcal{Q}^{(+1,-1)}_c\simeq\mathcal{Q}^{(+1,+1)}_c\) for even class \(c=2\ell+4\), \(0\le\ell\le u\). \\
The variant \(e=0\), resp. \(e=1\), is associated to the root \(R=\langle 243,6\rangle\), resp. \(R=\langle 243,8\rangle\).
\item
A parametrized family of \textbf{fork topologies}
for second \(3\)-class groups \(\mathrm{Gal}\left(F_3^{(2)}/F\right)\) of imaginary quadratic fields \(F\)
is given uniformly for the states \(\uparrow^\ell\) (ground state for \(\ell=0\), excited state for \(1 \le\ell\le u\))
of transfer kernel types in section \(\mathrm{E}\)
by the symmetric topology symbol
\begin{equation}
\label{eqn:ForkSecE}
P=
\overbrace{\mathrm{E}\stackrel{1}{\rightarrow}}^{\text{Leaf}}\quad
\overbrace{\left\lbrace\mathrm{c}\stackrel{1}{\rightarrow}\right\rbrace^{2\ell}\ }^{\text{Mainline}}\quad
\overbrace{\mathrm{c}}^{\text{Fork}}\quad
\overbrace{\left\lbrace\stackrel{2}{\leftarrow}\mathrm{c}\stackrel{1}{\leftarrow}\mathrm{c}\right\rbrace^{\ell}\ }^{\text{Maintrunk}}\quad
\overbrace{\stackrel{2}{\leftarrow}\mathrm{E}}^{\text{Leaf}}
\end{equation}
with scaffold type \(\mathrm{c}\) and the following invariants: \\
distance \(d=4\ell+2\)
(Dfn. \ref{dfn:VertexDistance}),
weighted distance \(w=5\ell+3\)
(Dfn. \ref{dfn:WeightedDistance}), \\
class increment \(\Delta\mathrm{cl}=(2\ell+5)-(2\ell+5)=0\),
coclass increment \(\Delta\mathrm{cc}=(\ell+3)-2=\ell+1\), \\
logarithmic order increment \(\Delta\mathrm{lo}=(3\ell+8)-(2\ell+7)=\ell+1\)
\cite[Dfn. 5.1, p. 89]{Ma15}.
\end{enumerate}
\end{theorem}

\begin{proof}
We compare the uniform generator rank \(d_1=2\) of all involved groups \(G_{c,j}^{(r)}\), \(c\ge 4\), \(r\ge 2\), \(1\le j\le 3\),
with their relation rank \(d_2\).
Since \(d_2=\mu\) and the \(p\)-multiplicator rank is \(\mu=2\) for \(S_{c,j}^{(r)}\) with odd \(c=2\ell+5\ge 5\) and \(r=\ell+3\ge 3\),
but \(\mu=3\) otherwise,
only the groups \(S_{c,j}^{(r)}\) are Schur \(\sigma\)-groups with balanced presentation \(d_2=2=d_1\),
and are therefore admissible as \(3\)-tower groups of imaginary quadratic fields \(F\),
according to our corrected version
\cite[\S\ 5, Thm. 5.1, pp. 28--29]{Ma10}
of the Shafarevich Theorem
\cite[Thm. 6, \((18^\prime)\)]{Sh}.
Finally we remark that
the nuclear rank is \(\nu=1\) for \(G_{c,j}^{(r)}\) with even \(c=2\ell+4\), \(r=\ell+2\), and child \(T_{c+1,j}^{(r)}\),
but \(\nu=0\) otherwise.

The execution of Algorithm
\ref{alg:ShafarevichCover}
with input data
\texttt{p:=3}, \texttt{vb:=25},
either \texttt{i:=6}, \texttt{e:=0},
or \texttt{i:=8}, \texttt{e:=1},
and \texttt{cp:=CompactPresentation(SmallGroup(243,i))},
proves the isomorphisms \(\mathcal{Q}^{(e,k)}_c\simeq S_{c,j}^{(\ell+3)}\), \(c=2\ell+5\),
resp. \(\mathcal{Q}^{(e,k)}_c\simeq G_{c,j}^{(\ell+2)}\), \(c=2\ell+4\),
for \(4\le c\le 20\), that is, \(0\le\ell\le u=8\).
The algorithm is initialized by the starting group \(R=M_3^{(2)}\) of coclass \(2\).
The loop navigates through the mainline vertices \(M_{c}^{(2)}\), \(c\ge 3\), of the coclass tree
\(\mathcal{T}^{(2)}\left(M_{3}^{(2)}\right)\).
The subroutine \texttt{IsAdmissible()} tests the transfer kernel type of all descendants and selects
either the unique capable descendant with type \(\mathrm{c}.18\), resp. \(\mathrm{c}.21\), for the flag \texttt{0},
or the unique descendant with type \(\mathrm{E}.6\), resp. \(\mathrm{E}.8\), for the flag \texttt{1},
or the first or second descendant with type \(\mathrm{E}.9\), for the flag \texttt{2}.
The selected non-mainline vertex is always
checked for isomorphism to the metabelianization of the appropriate quotient \(\mathcal{Q}^{(e,k)}_c\).
See also
\cite[\S\ 21.2, pp. 189--193]{Ma6},
\cite[pp. 751--756]{Ma8},
the proof of Theorem
\ref{thm:3ClassTowerLength3},
and Figures
\ref{fig:TreeTopoGStypeE},
\ref{fig:TreeTopoES1typeE},
\ref{fig:TreeTopoES2typeE}.
\end{proof}


Here again, a pure bottom up approach without top down constructions,
instead of using Algorithm
\ref{alg:ShafarevichCover},
is able to reach coclass \(r=32\), nilpotency class \(c=63\), and logarithmic order \(r+c=95\),
without surpassing internal limits of MAGMA,
and strongly supports Conjecture
\ref{cnj:ExplicitCovers}.

\begin{conjecture}
\label{cnj:ExplicitCovers}
Theorem
\ref{thm:ExplicitCovers}
remains true for any upper bound \(u>8\).
\end{conjecture}



\begin{figure}[ht]
\caption{Projections \(\mathcal{Q}^{(e,k)}_c\to\mathcal{Q}^{(e,k)}_c/\left(\mathcal{Q}^{(e,k)}_c\right)^{\prime\prime}\) of the covers onto their metabelianizations}
\label{fig:MetabProj}

\input{MetabProj}

\end{figure}



Figure
\ref{fig:MetabProj}
shows exactly the same situation as Figure
\ref{fig:GraphDissection},
supplemented by blue arrows indicating the projections
of the quotients \(\mathcal{Q}^{(e,k)}_c\) onto their metabelianizations,
that is, \(S_{c,j}^{(\ell+3)}\to G_{c,j}^{(2)}\), for odd class \(c=2\ell+5\), in the right diagram with green branches,
and \(G_{c,j}^{(\ell+2)}\to G_{c,j}^{(2)}\), for even class \(c=2\ell+4\), in the left diagram with red branches.
For \(c=4\), a degeneration occurs, since \(\mathcal{Q}^{(e,k)}_4\) is metabelian already,
indicated by surrounding blue circles.

Strictly speaking, the caption of Figure
\ref{fig:MetabProj}
in its full generality is valid for \(e=1\), \(M_{3}^{(2)}=\langle 243,8\rangle\) only.
For \(e=0\), \(M_{3}^{(2)}=\langle 243,6\rangle\),
all blue arrows have the same meaning as before
but the interpretation of the covers as quotients \(\mathcal{Q}^{(e,k)}_c\) is slightly restricted.
Whereas we have the following supplement to Formula
\eqref{eqn:IsomCovSecE}:
\begin{equation}
\label{eqn:IsomCovE14}
\text{type }\mathrm{E}.14:\ \mathcal{Q}^{(0,-1)}_c\simeq
\begin{cases}
S_{c,3}^{(\ell+3)} & \text{ for odd class } c=2\ell+5,\ 0\le\ell\le u, \\
G_{c,2}^{(\ell+2)} & \text{ for even class } c=2\ell+4,\ 0\le\ell\le u, \\
\end{cases}
\end{equation}
the quotients \(\mathcal{Q}^{(0,+1)}_c\) lead into a completely different realm,
namely the complicated brushwood of the complex transfer kernel type \(\mathrm{H}.4\).

Figure
\ref{fig:ComplexType}
shows three pruned descendant trees \(\mathcal{T}_\ast(\mathfrak{R})\) with roots
\(\mathfrak{R}=\langle 243,4\rangle\), \(\mathfrak{R}=\langle 6561,615\rangle\), and \(\mathfrak{R}=\langle 6561,613\rangle-\#1;1-\#2;1\),
all of whose vertices are of type \(\mathrm{H}.4\) exclusively.
We restrict the trees to \(\sigma\)-groups indicated by green color.
The top vertex \(\langle 27,3\rangle\) is intentionally drawn twice
to avoid an overlap of the dense trees and to admit a uniform representation of periodic bifurcations.
The tree with root \(\langle 243,4\rangle\) is not concerned by the quotients \(\mathcal{Q}^{(0,+1)}_c\).
It is sporadic and consists of periodically repeating finite saplings of depth \(2\) and increasing coclass \(2,3,\ldots\).
Connected by the maintrunk with vertices of type \(\mathrm{c}.18\) (red color)
in the descendant tree \(\mathcal{T}(\langle 243,6\rangle)\),
the trees with roots \(\langle 6561,615\rangle\) and \(\langle 6561,613\rangle-\#1;1-\#2;1\)
form the beginning of an infinite sequence of similar trees, which are, however, not isomorphic as graphs,
since the depth of the constituting saplings increases in steps of \(2\).
The projections of the quotients \(\mathcal{Q}^{(0,+1)}_c\) with odd class \(c\in\lbrace 5,7\rbrace\)
onto their metabelianizations are indicated by blue arrows.



\begin{figure}[ht]
\caption{Branches of \(\sigma\)-groups with complex type \(\mathrm{H}.4\) connected by the maintrunk}
\label{fig:ComplexType}

\input{ComplexType}

\end{figure}



\subsection{Topologies in descendant trees}
\label{ss:TreeTopologies}
\noindent
Tree topologies describe the mutual location
of distinct higher \(p\)-class groups \(\mathrm{G}_p^{(m)}{F}\) and \(\mathrm{G}_p^{(n)}{F}\),
with \(n>m\ge 1\),
of an algebraic number field \(F\).
The case \((m,n)=(3,4)\) will be crucial for finding the first examples of
\textit{four-stage towers} of \(p\)-class fields
with length \(\ell_p{F}:=\mathrm{dl}(\mathrm{G}_p^{(\infty)}{F})=4\),
which are unknown up to now, for any prime \(p\ge 2\).
Fork topologies with \((m,n)=(2,3)\) have proved to be essential for discovering
\(p\)-class towers with length \(\ell_p{F}=3\), for odd primes \(p\ge 3\).
In
\cite[Prp. 5.3, p. 89]{Ma15},
we have pointed out that the \textit{qualitative} topology problem for \((m,n)=(1,2)\)
is trivial, since the fork of \(\mathrm{G}_p^{(1)}{F}\) and \(\mathrm{G}_p^{(2)}{F}\)
is simply the abelian root \(\mathrm{G}_p^{(1)}{F}\simeq\mathrm{Cl}_p{F}\) of the entire relevant descendant tree.
However, the \textit{quantitative} structure of the root path between
\(\mathrm{G}_p^{(2)}{F}\) and \(\mathrm{G}_p^{(1)}{F}\) is not at all trivial
and can be given in a general theorem for \(\mathrm{Cl}_p{F}\simeq (p,p)\) and \(p\in\lbrace 2,3\rbrace\) only.
In the following Theorem
\ref{thm:MetabDescTopo},
we establish a purely group theoretic version of this result
by replacing \(\mathrm{G}_p^{(2)}{F}\) with an arbitrary \textit{metabelian} \(3\)-group \(\mathfrak{M}\)
having abelianization \(\mathfrak{M}/\mathfrak{M}^\prime\) of type \((3,3)\).
Any attempt to determine the group \(G:=\mathrm{Gal}(F_p^{(\infty)}/F)\)
of the \(p\)-class tower \(F_p^{(\infty)}\) of an algebraic number field \(F\)
begins with a search for the metabelianization \(\mathfrak{M}:=G/G^{\prime\prime}\),
i.e., the second derived quotient, of the \(p\)-tower group \(G\).
\(\mathfrak{M}\) is also called the \textit{second \(p\)-class group} \(\mathrm{Gal}(F_p^{(2)}/F)\) of \(F\),
and \(F_p^{(2)}\) can be viewed as a metabelian approximation of the \(p\)-class tower \(F_p^{(\infty)}\).
In the case of the smallest odd prime \(p=3\) and a number field \(F\)
with \(3\)-class group \(\mathrm{Cl}_3{F}\) of type \((3,3)\),
the structure of the root path from \(\mathfrak{M}\) to the root \(\langle 9,2\rangle\)
is known explicitly.
For its description, it suffices
to use the set of possible transfer kernel types 
\[\mathrm{X}\in\lbrace\mathrm{A},\mathrm{D},\mathrm{E},\mathrm{F},\mathrm{G},\mathrm{H},\mathrm{a},\mathrm{b},\mathrm{c},\mathrm{d}\rbrace\]
of the ancestors \(\pi^j\mathfrak{M}\), \(0\le j\le\ell\),
and the symbol \(\stackrel{s}{\to}\) for a weighted edge of step size \(s\ge 1\)
with formal exponents denoting iteration.
A capable vertex is indicated by an asterisk \(\mathrm{X}^\ast\).


\begin{theorem}
\label{thm:MetabDescTopo}
(Periodic root paths.) \\
There exist basically three kinds of root paths \(P:=\left(\pi^j\mathfrak{M}\right)_{0\le j\le\ell}\)
of metabelian \(3\)-groups \(\mathfrak{M}\) with abelianization \(\mathfrak{M}/\mathfrak{M}^\prime\) of type \((3,3)\),
which are located on coclass trees.
Let \(c\) denote the nilpotency class \(\mathrm{cl}(\mathfrak{M})\) and
\(r\) the coclass \(\mathrm{cc}(\mathfrak{M})\) of \(\mathfrak{M}\).
\begin{enumerate}
\item
If \(r=1\) and \(c\ge 1\), then \(P=\mathrm{X}\left\lbrace\stackrel{1}{\to}\mathrm{a}^\ast\right\rbrace^{c-1}\),
where \(\mathrm{X}\in\lbrace\mathrm{A},\mathrm{a},\mathrm{a}^\ast\rbrace\).
\item
If \(r=2\) and \(c\ge 3\), then
either \(P=\mathrm{X}
\left\lbrace\stackrel{1}{\to}\mathrm{b}^\ast\right\rbrace^{c-3}
\stackrel{2}{\to}\mathrm{a}^\ast\stackrel{1}{\to}\mathrm{a}^\ast\),
where \(\mathrm{X}\in\lbrace\mathrm{d},\mathrm{b},\mathrm{b}^\ast\rbrace\),
or \(P=\mathrm{X}
\left\lbrace\stackrel{1}{\to}\mathrm{c}^\ast\right\rbrace^{c-3}
\stackrel{2}{\to}\mathrm{a}^\ast\stackrel{1}{\to}\mathrm{a}^\ast\),
where \(\mathrm{X}\in\lbrace\mathrm{E},\mathrm{G}^\ast,\mathrm{H}^\ast,\mathrm{c}^\ast\rbrace\).
An additional variant arises for \(r=2\), \(c\ge 5\), with
\(P=\mathrm{X}
\stackrel{1}{\to}\mathrm{X}^\ast
\left\lbrace\stackrel{1}{\to}\mathrm{c}^\ast\right\rbrace^{c-4}
\stackrel{2}{\to}\mathrm{a}^\ast\stackrel{1}{\to}\mathrm{a}^\ast\),
where \(\mathrm{X}\in\lbrace\mathrm{G},\mathrm{H}\rbrace\).
\item
If \(r\ge 3\) and \(c\ge r+1\), then
either \(P=\mathrm{X}
\left\lbrace\stackrel{1}{\to}\mathrm{b}^\ast\right\rbrace^{c-(r+1)}
\left\lbrace\stackrel{2}{\to}\mathrm{b}^\ast\right\rbrace^{r-2}
\stackrel{2}{\to}\mathrm{a}^\ast\stackrel{1}{\to}\mathrm{a}^\ast\),
where \(\mathrm{X}\in\lbrace\mathrm{d},\mathrm{b},\mathrm{b}^\ast\rbrace\),
or \(P=\mathrm{X}
\left\lbrace\stackrel{1}{\to}\mathrm{d}^\ast\right\rbrace^{c-(r+1)}
\left\lbrace\stackrel{2}{\to}\mathrm{b}^\ast\right\rbrace^{r-2}
\stackrel{2}{\to}\mathrm{a}^\ast\stackrel{1}{\to}\mathrm{a}^\ast\),
where \(\mathrm{X}\in\lbrace\mathrm{F},\mathrm{G}^\ast,\mathrm{H}^\ast,\mathrm{d}^\ast\rbrace\).
An additional variant arises for \(r\ge 3\), \(c\ge r+3\), with \\
\(P=\mathrm{X}
\stackrel{1}{\to}\mathrm{X}^\ast
\left\lbrace\stackrel{1}{\to}\mathrm{d}^\ast\right\rbrace^{c-(r+2)}
\left\lbrace\stackrel{2}{\to}\mathrm{b}^\ast\right\rbrace^{r-2}
\stackrel{2}{\to}\mathrm{a}^\ast\stackrel{1}{\to}\mathrm{a}^\ast\),
where \(\mathrm{X}\in\lbrace\mathrm{G},\mathrm{H}\rbrace\).
\end{enumerate}
In particular, the maximal possible step size is \(s=2\),
and the \(r-1\) edges with step size \(s=2\) arise successively without gaps at the end of the path,
except the trailing edge of step size \(s=1\).
\end{theorem}

\begin{proof}
\(\mathrm{X}\) always denotes the type of the starting vertex \(\mathfrak{M}\).
The remaining vertices of the root path form the \textit{scaffold},
which connects the starting vertex with the ending vertex (the root \(R=\langle 9,2\rangle\)).
The unique coclass tree \(\mathcal{T}^{(1)}\langle 9,2\rangle\) with \(r=1\) has a mainline of type \(\mathrm{a}^\ast\).
Two of the coclass trees \(\mathcal{T}^{(2)}\langle 243,n\rangle\) with \(r=2\), those with \(n\in\lbrace 6,8\rbrace\),
have mainlines of type \(\mathrm{c}^\ast\) and an additional scaffold of type  \(\mathrm{a}^\ast\).
For \(n=3\), the mainline is of type \(\mathrm{b}^\ast\).
The coclass trees \(\mathcal{T}^{(r)}\) with \(r\ge 3\) behave uniformly
with mainlines of type \(\mathrm{b}^\ast\) or \(\mathrm{d}^\ast\) and scaffold types \(\mathrm{b}^\ast\), \(\mathrm{a}^\ast\).
For details, see Nebelung
\cite[Satz 3.3.7, p. 70, Lemma 5.2.6, p. 183, Satz 6.9, p. 202, Satz 6.14, p. 208]{Ne}.
\end{proof}

\begin{remark}
\label{rmk:MetabDescTopo}
The final statement of Theorem
\ref{thm:MetabDescTopo}
is a graph theoretic reformulation of the quotient structure
of the lower central series \(\left(\gamma_j{\mathfrak{M}}\right)_{j\ge 1}\)
of a metabelian \(3\)-group \(\mathfrak{M}\),
observing that the root \(R=\langle 9,2\rangle\) corresponds to the bicyclic quotient \(\gamma_1/\gamma_2\simeq (3,3)\)
and the conspicuous trailing edge \(\stackrel{1}{\to}\mathrm{a}^\ast\) corresponds to the cyclic bottleneck \(\gamma_2/\gamma_3\simeq (3)\).
The structure is drawn ostensively in Formula (2.12) of
\cite[\S\ 2.2]{Ma12},
using the CF-invariant \(e=r+1\) instead of the coclass \(r\).
\end{remark}


Theorem
\ref{thm:MetabDescTopo}
concerns periodic vertices on coclass trees.
Sporadic vertices outside of coclass trees
must be treated separately in Corollary
\ref{cor:MetabDescTopo}.

\begin{corollary}
\label{cor:MetabDescTopo}
(Sporadic root paths.) \\
As before, let \(\mathfrak{M}\) be a metabelian \(3\)-group
with abelianization \(\mathfrak{M}/\mathfrak{M}^\prime\simeq(3,3)\),
nilpotency class \(c:=\mathrm{cl}(\mathfrak{M})\) and
coclass \(r:=\mathrm{cc}(\mathfrak{M})\).
Assume that \(\mathfrak{M}\) is located outside of coclass trees.
\begin{enumerate}
\item
If \(r=2\) and \(c=3\), then
\(P=\mathrm{X}
\stackrel{2}{\to}\mathrm{a}^\ast\stackrel{1}{\to}\mathrm{a}^\ast\),
where \(\mathrm{X}\in\lbrace\mathrm{D},\mathrm{G}^\ast,\mathrm{H}^\ast\rbrace\).
\item
If \(r=2\) and \(c=4\), then
\(P=\mathrm{X}
\stackrel{1}{\to}\mathrm{X}^\ast
\stackrel{2}{\to}\mathrm{a}^\ast\stackrel{1}{\to}\mathrm{a}^\ast\),
where \(\mathrm{X}\in\lbrace\mathrm{G},\mathrm{H}\rbrace\).
\item
If \(r\ge 3\) and \(c=r+1\), then
\(P=\mathrm{X}
\left\lbrace\stackrel{2}{\to}\mathrm{b}^\ast\right\rbrace^{r-2}
\stackrel{2}{\to}\mathrm{a}^\ast\stackrel{1}{\to}\mathrm{a}^\ast\),
where \(\mathrm{X}\in\lbrace\mathrm{F},\mathrm{G}^\ast,\mathrm{H}^\ast\rbrace\).
\item
If \(r\ge 3\) and \(c=r+2\), then
\(P=\mathrm{X}
\stackrel{1}{\to}\mathrm{X}^\ast
\left\lbrace\stackrel{2}{\to}\mathrm{b}^\ast\right\rbrace^{r-2}
\stackrel{2}{\to}\mathrm{a}^\ast\stackrel{1}{\to}\mathrm{a}^\ast\),
where \(\mathrm{X}\in\lbrace\mathrm{G},\mathrm{H}\rbrace\).
\end{enumerate}
\end{corollary}

\begin{proof}
As in the proof of Theorem
\ref{thm:MetabDescTopo},
see the dissertation of Nebelung
\cite{Ne}.
\end{proof}



\subsection{Computing Artin patterns of \(p\)-groups}
\label{ss:TransferKernelType}
\noindent
In both Algorithms
\ref{alg:Mainlines}
and
\ref{alg:ShafarevichCover},
we made use of a subroutine \texttt{IsAdmissible()}
which filters \(p\)-groups \(G\) with abelianization \(G/G^\prime\simeq (p,p)\)
having a prescribed transfer kernel type (TKT).
Since an algorithm of this kind is not implemented in MAGMA,
we briefly communicate a succinct form of the code for this subroutine.

\begin{algorithm}
\label{alg:TransferKernelType}
(Transfer kernel type.) \\
\textbf{Input:} a prime number \(p\) and a finite \(p\)-group \texttt{G}.\\
\textbf{Code:}
\texttt{
\begin{tabbing}
   for \= for \= for \= for \= for \= for \= \kill
         if ([p,p] eq AbelianQuotientInvariants(G)) then\+\\
            x := G.1;
            y := G.2; // main generators\\
            A := [];
            B := []; // generators and transversal\\
            Append(\(\sim\)A,y);\\
            Append(\(\sim\)B,x);\\
            for e in [0..p-1] do\+\\
               Append(\(\sim\)A,x\(\ast\)y\({}\,\hat{}\,{}\)e);\\
               Append(\(\sim\)B,y);\-\\
            end for;\\
            DG := DerivedSubgroup(G);\\
            nTotal := 0;
            nFixed := 0;\\
            TKT := [];\\
            for i in [1..p+1] do\+\\
               M := sub<G|A[i],DG>;\\
               DM := DerivedSubgroup(M);\\
               AQM,pr := M/DM;\\
               ImA := (A[i]\(\ast\)B[i]\({}\,\hat{}\,{}\)-1)\({}\,\hat{}\,{}\)p\(\ast\)B[i]\({}\,\hat{}\,{}\)p; // inner transfer\\
               ImB := B[i]\({}\,\hat{}\,{}\)p; // outer transfer\\
               T := hom<G->AQM|<A[i],(ImA)@pr>,<B[i],(ImB)@pr>>;\\
               KT := sub<G|DG,Kernel(T)>;\\
               if KT eq G then // total kernel\+\\
                  Append(\(\sim\)TKT,0);\\
                  nTotal := nTotal+1;\-\\
               else\+\\
                  for j in [1..p+1] do\+\\
                     if A[j] in KT then\+\\
                        Append(\(\sim\)TKT,j);\\
                        if (i eq j) then // fixed point\+\\
                           nFixed := nFixed+1;\-\\
                        end if;\-\\
                     end if;\-\\
                  end for;\-\\
               end if;\-\\
            end for;\\
            image := [];\\
            for i in [1..p+1] do\+\\
               if not (TKT[i] in image) then\+\\
                  Append(\(\sim\)image,TKT[i]);\-\\
               end if;\-\\
            end for;\\
            occupation := \(\#\)image;\\
            repetitions := 0; // maximal occupation number\\
            intersection := 0; // meet of repetitions and fixed points\\                     
            doublet := 0;\\
            for digit in [1..p+1] do\+\\
               counter := 0;\\
               for j in [1..\(\#\)TKT] do\+\\
                  if (digit eq TKT[j]) then\+\\
                     counter := counter + 1;\-\\
                  end if;\-\\
               end for;\\
               if (counter ge 2) then\+\\
                  doublet := digit;\-\\
               end if;\\
               if (counter gt repetitions) then\+\\
                  repetitions := counter;\-\\
               end if;\-\\
            end for;\\
            if (doublet ge 1) then\+\\
               if (doublet eq TKT[doublet]) then\+\\
                  intersection := 1;\-\\
               end if;\-\\
            end if;\-\\
         end if;
\end{tabbing}
}
\noindent
\textbf{Output:}
transfer kernel type \texttt{TKT}, number \texttt{nTotal} of total kernels, number \texttt{nFixed} of fixed points,
and further invariants \texttt{occupation,repetitions,intersection} describing the orbit of the TKT.
\end{algorithm}

The output of Algorithm
\ref{alg:TransferKernelType}
is used for the subroutine \texttt{IsAdmissible(G,p,t)} in dependence on the parameter flag \texttt{t}.
When the root \(R=\langle 243,8\rangle\) is selected for the tree \(\mathcal{T}(R)\)
the return value is determined in the following manner:
\texttt{
\begin{tabbing}
   for \= for \kill
            if (0 eq t) then\+\\
               return ((1 eq nTotal) and (2 eq nFixed)); // type c.21\-\\
            elif (1 eq t) then\+\\
               return ((0 eq nTotal) and (3 eq nFixed)); // type E.8\-\\
            elif (2 eq t) then\+\\
               return ((0 eq nTotal) and (2 eq nFixed) and (3 eq occupation)); // type E.9\-\\
            end if;
\end{tabbing}
}
For the root \(R=\langle 243,6\rangle\), we have
\texttt{
\begin{tabbing}
   for \= for \kill
            if (0 eq t) then\+\\
               return ((1 eq nTotal) and (0 eq nFixed)); // type c.18\-\\
            elif (1 eq t) then\+\\
               return ((0 eq nTotal) and (1 eq nFixed)); // type E.6\-\\
            elif (2 eq t) then\+\\
               return ((0 eq nTotal) and (0 eq nFixed) and (3 eq occupation)); // type E.14\-\\
            end if;
\end{tabbing}
}



\subsection{Benefits and drawbacks of bottom up and top down techniques}
\label{ss:BenefitDrawback}
\noindent
In this chapter, we have presented several convenient ways
of expressing information about \textit{infinite sequences} of finite \(p\)-groups.
Each of them has its benefits and drawbacks.

The \textit{bottom up strategy} of constructing finite \(p\)-groups
as successive extensions of a (metabelian or even abelian) starting group \(R\),
called the \textit{root},
by recursive applications of the \(p\)-group algorithm by Newman
\cite{Nm}
and O'Brien
\cite{OB}
has the benefit of visualizing the graph theoretic \textit{root path} in the descendant tree \(\mathcal{T}(R)\).
Its implementation in MAGMA is incredibly stable and robust
without surpassing any internal limits up to logarithmic orders of \(95\) and even more.
Only the consumption of CPU time becomes considerable in such extreme regions.

The \textit{top down strategy} of expressing finite \(p\)-groups
as quotients of an infinite pro-\(p\) group with given pro-\(p\) presentation
has the benefit of including non-metabelian groups with arbitrary coclass \(r\ge 3\),
periodic mainline vertices in Algorithm
\ref{alg:Mainlines}
and sporadic Schur \(\sigma\)-leaves in Algorithm
\ref{alg:ShafarevichCover}.
The drawback is that the evaluation of the pro-\(p\) presentation in MAGMA
exceeds the maximal permitted word length for nilpotency class \(c\ge 36\).

Up to this point, we have not yet touched upon \textit{parametrized polycyclic power-commutator presentations}
\cite{Ma5c}.
For the root \(R=\langle 243,6\rangle\),
the metabelian vertices \(G\) of the coclass tree \(\mathcal{T}^{(2)}(R)\)
with class \(c=\mathrm{cl}(G)\ge 5\), down to depth \(\mathrm{dp}(G)\le 1\),
can be presented in the form
\begin{equation}
\label{eqn:PolycyclicPresentation}
\begin{aligned}
G_c(\alpha,\beta) = \langle\ & x,y,s_2,t_3,s_3,\ldots,s_c\ \mid \\
                  & s_2=\lbrack y,x\rbrack,\ t_3=\lbrack s_2,y\rbrack,\ s_i=\lbrack s_{i-1},x\rbrack \text{ for } 3\le i\le c, \\
                  & x^3=s_c^\alpha,\ y^3s_3^{-2}s_4^{-1}=s_c^\beta,\ s_i^3=s_{i+2}^2s_{i+3} \text{ for } 2\le i\le c-3,\ s_{c-2}^3=s_c^2\ \rangle,
\end{aligned}
\end{equation}
where the parameters \(\alpha\) and \(\beta\) depend on the transfer kernel type \(\varkappa(G)\),
\begin{equation}
\label{eqn:ParametrizedPresentation}
(\alpha,\beta)=
\begin{cases}
(0,0) & \text{ for } \varkappa(G)\sim (0122),\ \mathrm{c}.18, \\
(1,0) & \text{ for } \varkappa(G)\sim (1122),\ \mathrm{E}.6, \\
(0,1) \text{ or } (0,2) & \text{ for } \varkappa(G)\sim (2122),\ \mathrm{H}.4, \\
(1,1) \text{ or } (1,2) & \text{ for } \varkappa(G)\sim (3122)\sim (4122),\ \mathrm{E}.14.
\end{cases}
\end{equation}
This presentation has the benefit of including six periodic sequences with distinct transfer kernel types,
and the drawback of being restricted to the fixed coclass \(2\).



\section{The first \(3\)-class towers of length \(3\)}
\label{s:3ClassTowerLength3}
\noindent
In our long desired disproof of the claim by Scholz and Taussky
\cite[p. 41]{SoTa}
concerning the \(3\)-class tower of the imaginary quadratic field \(F=\mathbb{Q}(\sqrt{-9\,748})\),
we presented the first \(p\)-class towers with exactly three stages, for an odd prime \(p\),
in cooperation with Bush
\cite[Cor. 4.1.1, p. 775]{BuMa}.
The underlying fields \(F\) were of type \(\mathrm{E}.9\) in its ground state,
which admits two possibilities for the second \(3\)-class group,
\(\mathfrak{M}\simeq\langle 2187,302\rangle\) or \(\langle 2187,306\rangle\).
Now we want to illustrate the way which led to the \textit{fork topologies} in Theorem
\ref{thm:ExplicitCovers}
by using the more convenient type \(\mathrm{E}.8\),
where the group \(\mathfrak{M}\) is unique for every state,
in particular, \(\mathfrak{M}\simeq\langle 2187,304\rangle\) for the ground state.

\begin{remark}
\label{rmk:LogInv}
Concerning the notation, we are going to use
\textit{logarithmic type invariants} of abelian \(3\)-groups, for instance
\((21)\hat{=}(9,3)\), \((32)\hat{=}(27,9)\), \((43)\hat{=}(81,27)\), and \((54)\hat{=}(243,81)\). 
\end{remark}

\noindent
Let \(F=\mathbb{Q}(\sqrt{d})\) be an \textit{imaginary} quadratic number field
with \(3\)-class group \(\mathrm{Cl}_3{F}\simeq (3,3)\), and let
\(E_1,\ldots,E_4\)
be the unramified cyclic cubic extensions of \(F\).


\begin{theorem}
\label{thm:3ClassTowerLength3} 
(First towers of type \(\mathrm{E}.8\).) \quad
Let the capitulation of \(3\)-classes of \(F\) in \(E_1,\ldots,E_4\)
be of type \(\varkappa_1{F}\sim (\mathbf{1,2,3},1)\), which is called type \(\mathrm{E}.8\).
Assume further that the \(3\)-class groups of \(E_1,\ldots,E_4\)
possess the abelian type invariants \(\tau_1{F}\sim\lbrack T_1,21,21,21\rbrack\),
where \(T_1\in\lbrace 32,43,54\rbrace\).

Then the length of the \(3\)-class tower of \(F\) is precisely \(\ell_3{F}=3\).
\end{theorem}

\begin{proof}
We employ the \(p\)-group generation algorithm
\cite{Nm,OB}
for searching the Artin pattern \(\mathrm{AP}(F)=(\tau_1{F},\varkappa_1{F})\)
among the descendants of the root \(R:=C_3\times C_3=\langle 9,2\rangle\) in the tree \(\mathcal{T}(R)\).
After two steps, \(\langle 9,2\rangle\leftarrow\langle 27,3\rangle\leftarrow\langle 243,8\rangle\),
we find the next root \(U_5:=\langle 243,8\rangle\) of the unique relevant coclass tree \(\mathcal{T}^{(2)}(U_5)\),
using the assigned simple TKT E.8, \(\varkappa_3=(1231)\), and its associated scaffold TKT c.21, \(\varkappa_0=(0231)\).
Finally, the first component \(T_1=\tau_1(1)\in\lbrace 32,43,54\rbrace\) of the TTT
provides the break-off condition, according to
\cite[Thm. 1.21, p. 79]{Ma15},
resp. Theorem M in
\cite[p. 14]{Ma15b},
and we get
\(\mathfrak{M}\simeq\langle 2187,304\rangle=\langle 729,54\rangle-\#1;4\)
for the ground state \(T_1=(32)\),
\(\mathfrak{M}\simeq\langle 6561,2050\rangle-\#1;2\)
for the \(1\)st excited state \(T_1=(43)\), and
\(\mathfrak{M}\simeq\langle 6561,2050\rangle(-\#1;1)^2-\#1;2\)
for the \(2\)nd excited state \(T_1=(54)\),
where \(\langle 2187,303\rangle=\langle 729,54\rangle-\#1;3\) and \(\langle 6561,2050\rangle=\langle 2187,303\rangle-\#1;1\).
The situation is visualized by Figure
\ref{fig:TreeUMinDisc},
where the three metabelianizations \(\mathfrak{M}\simeq G/G^{\prime\prime}\) of the \(3\)-tower group \(G\),
for the ground state and two excited states,
are emphasized with red color.
Figure
\ref{fig:TreeUMinDisc},
showing the second \(3\)-class groups \(\mathfrak{M}\),
was essentially known to Ascione in 1979
\cite{As1,As2},
and to Nebelung in 1989
\cite{Ne}.
Compare the historical remarks
\cite[\S\ 3, p. 163]{Ma6}.

In the next three Figures
\ref{fig:TreeTopoGStypeE},
\ref{fig:TreeTopoES1typeE},
and
\ref{fig:TreeTopoES2typeE},
which were unknown until 2012,
we present the decisive break-through
establishing the first rigorous proof for three-stage towers of \(3\)-class fields.
The key ingredient is the discovery of periodic bifurcations
\cite[\S\ 3, p. 163]{Ma6}
in the complete descendant tree \(\mathcal{T}(U_5)\) which is of considerably higher complexity
than the coclass tree  \(\mathcal{T}^{(2)}(U_5)\).

For the ground state \(T_1=(32)\),
the first bifurcation yields the cover
\[\mathrm{cov}(\mathfrak{M})=\lbrace\mathfrak{M},\langle 6561,622\rangle\rbrace\]
of
\(\mathfrak{M}\simeq\langle 2187,304\rangle\).
The relation rank \(d_2{\mathfrak{M}}=3\) eliminates \(\mathfrak{M}\)
as a candidate for the \(3\)-tower group \(G\),
according to the Corollary
\cite[p. 7]{Ma15b}
of the Shafarevich Theorem
\cite[Thm. 1.3, pp. 75--76]{Ma15},
and we end up getting \(G\simeq\langle 6561,622\rangle=\langle 729,54\rangle-\#2;4\)
with a siblings topology
\[\mathrm{E}\stackrel{1}{\rightarrow}\ 
\mathrm{c}\ 
\stackrel{2}{\leftarrow}\mathrm{E}\]
which describes the relative location of \(\mathfrak{M}\) and \(G\).

For the first excited state \(T_1=(43)\),
the second bifurcation yields the cover
\[\mathrm{cov}(\mathfrak{M})=\lbrace\mathfrak{M},\langle 6561,621\rangle-\#1;1-\#1;2,\langle 6561,621\rangle-\#1;1-\#2;2\rbrace\]
of
\(\mathfrak{M}\simeq\langle 6561,2050\rangle-\#1;2\),
where \(\langle 6561,621\rangle=\langle 729,54\rangle-\#2;3\).
The relation rank \(d_2=3\) eliminates \(\mathfrak{M}\) and \(\langle 6561,621\rangle-\#1;1-\#1;2\)
as candidates for the \(3\)-tower group \(G\),
according to Shafarevich,
and we get the unique \(G\simeq\langle 6561,621\rangle-\#1;1-\#2;2\)
with fork topology
\[\mathrm{E}\stackrel{1}{\rightarrow}\ 
\left\lbrace\mathrm{c}\stackrel{1}{\rightarrow}\right\rbrace^{2}\
\mathrm{c}\ 
\left\lbrace\stackrel{2}{\leftarrow}\mathrm{c}\stackrel{1}{\leftarrow}\mathrm{c}\right\rbrace\
\stackrel{2}{\leftarrow}\mathrm{E}.\]

Similarly, the second excited state \(T_1=(54)\)
yields a more complex advanced fork topology
\[\mathrm{E}\stackrel{1}{\rightarrow}\ 
\left\lbrace\mathrm{c}\stackrel{1}{\rightarrow}\right\rbrace^{4}\
\mathrm{c}\ 
\left\lbrace\stackrel{2}{\leftarrow}\mathrm{c}\stackrel{1}{\leftarrow}\mathrm{c}\right\rbrace^{2}\
\stackrel{2}{\leftarrow}\mathrm{E}.\qedhere\]
\end{proof}



{\tiny

\vspace{0.2in}

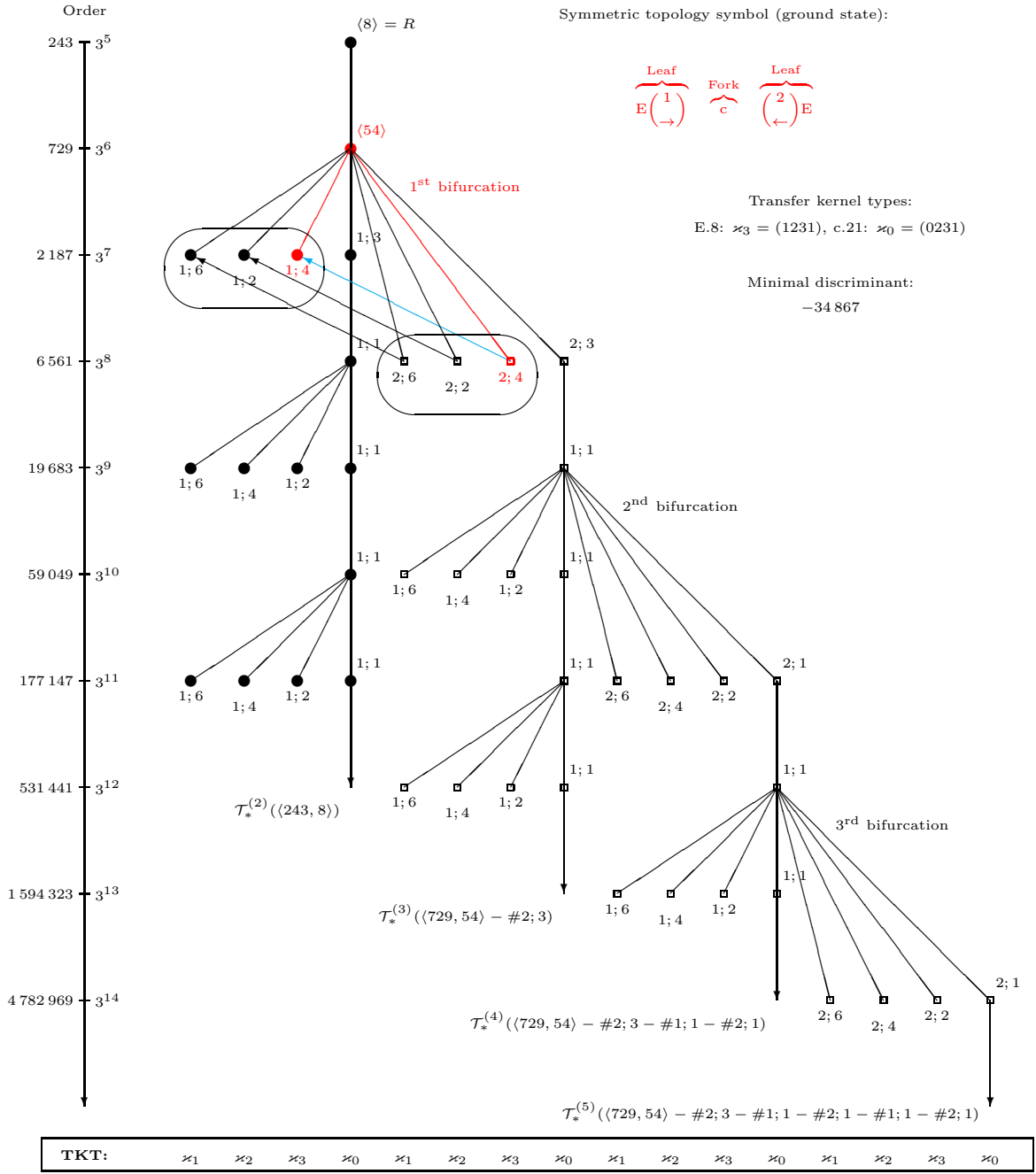
\begin{figure}[ht]
\caption{Tree topology of type E in the ground state}
\label{fig:TreeTopoGStypeE}


\setlength{\unitlength}{0.8cm}
\begin{picture}(18,22)(-6,-21)


\put(-5,0.5){\makebox(0,0)[cb]{Order}}
\put(-5,0){\line(0,-1){18}}
\multiput(-5.1,0)(0,-2){10}{\line(1,0){0.2}}
\put(-5.2,0){\makebox(0,0)[rc]{\(243\)}}
\put(-4.8,0){\makebox(0,0)[lc]{\(3^5\)}}
\put(-5.2,-2){\makebox(0,0)[rc]{\(729\)}}
\put(-4.8,-2){\makebox(0,0)[lc]{\(3^6\)}}
\put(-5.2,-4){\makebox(0,0)[rc]{\(2\,187\)}}
\put(-4.8,-4){\makebox(0,0)[lc]{\(3^7\)}}
\put(-5.2,-6){\makebox(0,0)[rc]{\(6\,561\)}}
\put(-4.8,-6){\makebox(0,0)[lc]{\(3^8\)}}
\put(-5.2,-8){\makebox(0,0)[rc]{\(19\,683\)}}
\put(-4.8,-8){\makebox(0,0)[lc]{\(3^9\)}}
\put(-5.2,-10){\makebox(0,0)[rc]{\(59\,049\)}}
\put(-4.8,-10){\makebox(0,0)[lc]{\(3^{10}\)}}
\put(-5.2,-12){\makebox(0,0)[rc]{\(177\,147\)}}
\put(-4.8,-12){\makebox(0,0)[lc]{\(3^{11}\)}}
\put(-5.2,-14){\makebox(0,0)[rc]{\(531\,441\)}}
\put(-4.8,-14){\makebox(0,0)[lc]{\(3^{12}\)}}
\put(-5.2,-16){\makebox(0,0)[rc]{\(1\,594\,323\)}}
\put(-4.8,-16){\makebox(0,0)[lc]{\(3^{13}\)}}
\put(-5.2,-18){\makebox(0,0)[rc]{\(4\,782\,969\)}}
\put(-4.8,-18){\makebox(0,0)[lc]{\(3^{14}\)}}
\put(-5,-18){\vector(0,-1){2}}


\put(7,0.5){\makebox(0,0)[cc]{Symmetric topology symbol (ground state):}}
{\color{red}
\put(7,-1){\makebox(0,0)[cc]{
\(\overbrace{\mathrm{E}\binom{1}{\rightarrow}}^{\text{Leaf}}\quad
\overbrace{\mathrm{c}}^{\text{Fork}}\quad
\overbrace{\binom{2}{\leftarrow}\mathrm{E}}^{\text{Leaf}}\)
}}
}
\put(9,-3){\makebox(0,0)[cc]{Transfer kernel types:}}
\put(9,-3.5){\makebox(0,0)[cc]{E.8: \(\varkappa_3=(1231)\), c.21: \(\varkappa_0=(0231)\)}}
\put(9,-4.5){\makebox(0,0)[cc]{Minimal discriminant:}}
\put(9,-5){\makebox(0,0)[cc]{\(-34\,867\)}}

\put(0.1,0.2){\makebox(0,0)[lb]{\(\langle 8\rangle=R\)}}
{\color{red}
\put(0.1,-1.8){\makebox(0,0)[lb]{\(\langle 54\rangle\)}}
\put(1.1,-2.8){\makebox(0,0)[lb]{\(1^{\text{st}}\) bifurcation}}
}
\put(0.1,-3.8){\makebox(0,0)[lb]{\(1;3\)}}
\put(0.1,-5.8){\makebox(0,0)[lb]{\(1;1\)}}
\put(0.1,-7.8){\makebox(0,0)[lb]{\(1;1\)}}
\put(0.1,-9.8){\makebox(0,0)[lb]{\(1;1\)}}
\put(0.1,-11.8){\makebox(0,0)[lb]{\(1;1\)}}
\put(0,0){\circle*{0.2}}
{\color{red}
\put(0,-2){\circle*{0.2}}
}
\multiput(0,-4)(0,-2){5}{\circle*{0.2}}
\multiput(0,0)(0,-2){6}{\line(0,-1){2}}
\put(0,-12){\vector(0,-1){2}}
\put(-0.2,-14.2){\makebox(0,0)[rt]{\(\mathcal{T}_\ast^{(2)}(\langle 243,8\rangle)\)}}

\put(-3,-4.2){\makebox(0,0)[ct]{\(1;6\)}}
\put(-3,-8.2){\makebox(0,0)[ct]{\(1;6\)}}
\put(-3,-12.2){\makebox(0,0)[ct]{\(1;6\)}}
\multiput(0,-2)(0,-4){3}{\line(-3,-2){3}}
\multiput(-3,-4)(0,-4){3}{\circle*{0.2}}

\put(-2,-4.4){\makebox(0,0)[ct]{\(1;2\)}}
\put(-2,-8.4){\makebox(0,0)[ct]{\(1;4\)}}
\put(-2,-12.4){\makebox(0,0)[ct]{\(1;4\)}}
\multiput(0,-2)(0,-4){3}{\line(-1,-1){2}}
\multiput(-2,-4)(0,-4){3}{\circle*{0.2}}

{\color{red}
\put(-1,-4.2){\makebox(0,0)[ct]{\(1;4\)}}
}
\put(-1,-8.2){\makebox(0,0)[ct]{\(1;2\)}}
\put(-1,-12.2){\makebox(0,0)[ct]{\(1;2\)}}
{\color{red}
\put(0,-2){\line(-1,-2){1}}
\put(-1,-4){\circle*{0.2}}
}
\multiput(0,-6)(0,-4){2}{\line(-1,-2){1}}
\multiput(-1,-8)(0,-4){2}{\circle*{0.2}}



\multiput(1,-6)(1,0){2}{\vector(-2,1){3.9}}
{\color{cyan}
\multiput(3,-6)(1,0){1}{\vector(-2,1){3.9}}
}

\put(0,-2){\line(1,-1){4}}

\put(4.1,-5.8){\makebox(0,0)[lb]{\(2;3\)}}
\put(4.1,-7.8){\makebox(0,0)[lb]{\(1;1\)}}
\put(5.1,-8.8){\makebox(0,0)[lb]{\(2^{\text{nd}}\) bifurcation}}
\put(4.1,-9.8){\makebox(0,0)[lb]{\(1;1\)}}
\put(4.1,-11.8){\makebox(0,0)[lb]{\(1;1\)}}
\put(4.1,-13.8){\makebox(0,0)[lb]{\(1;1\)}}
\multiput(3.95,-6.05)(0,-2){5}{\framebox(0.1,0.1){}}
\multiput(4,-6)(0,-2){4}{\line(0,-1){2}}
\put(4,-14){\vector(0,-1){2}}
\put(3.8,-16.2){\makebox(0,0)[rt]{\(\mathcal{T}_\ast^{(3)}(\langle 729,54\rangle-\#2;3)\)}}

\put(1,-6.2){\makebox(0,0)[ct]{\(2;6\)}}
\put(1,-10.2){\makebox(0,0)[ct]{\(1;6\)}}
\put(1,-14.2){\makebox(0,0)[ct]{\(1;6\)}}
\put(0,-2){\line(1,-4){1}}
\multiput(4,-8)(0,-4){2}{\line(-3,-2){3}}
\multiput(0.95,-6.05)(0,-4){3}{\framebox(0.1,0.1){}}

\put(2,-6.4){\makebox(0,0)[ct]{\(2;2\)}}
\put(2,-10.4){\makebox(0,0)[ct]{\(1;4\)}}
\put(2,-14.4){\makebox(0,0)[ct]{\(1;4\)}}
\put(0,-2){\line(1,-2){2}}
\multiput(4,-8)(0,-4){2}{\line(-1,-1){2}}
\multiput(1.95,-6.05)(0,-4){3}{\framebox(0.1,0.1){}}

{\color{red}
\put(3,-6.2){\makebox(0,0)[ct]{\(2;4\)}}
}
\put(3,-10.2){\makebox(0,0)[ct]{\(1;2\)}}
\put(3,-14.2){\makebox(0,0)[ct]{\(1;2\)}}
{\color{red}
\put(0,-2){\line(3,-4){3}}
\put(2.95,-6.05){\framebox(0.1,0.1){}}
}
\multiput(4,-8)(0,-4){2}{\line(-1,-2){1}}
\multiput(2.95,-10.05)(0,-4){2}{\framebox(0.1,0.1){}}

\put(4,-8){\line(1,-1){4}}

\put(8.1,-11.8){\makebox(0,0)[lb]{\(2;1\)}}
\put(8.1,-13.8){\makebox(0,0)[lb]{\(1;1\)}}
\put(9.1,-14.8){\makebox(0,0)[lb]{\(3^{\text{rd}}\) bifurcation}}
\put(8.1,-15.8){\makebox(0,0)[lb]{\(1;1\)}}
\multiput(7.95,-12.05)(0,-2){3}{\framebox(0.1,0.1){}}
\multiput(8,-12)(0,-2){2}{\line(0,-1){2}}
\put(8,-16){\vector(0,-1){2}}
\put(7.8,-18.2){\makebox(0,0)[rt]{\(\mathcal{T}_\ast^{(4)}(\langle 729,54\rangle-\#2;3-\#1;1-\#2;1)\)}}

\put(5,-12.2){\makebox(0,0)[ct]{\(2;6\)}}
\put(5,-16.2){\makebox(0,0)[ct]{\(1;6\)}}
\put(4,-8){\line(1,-4){1}}
\multiput(8,-14)(0,-4){1}{\line(-3,-2){3}}
\multiput(4.95,-12.05)(0,-4){2}{\framebox(0.1,0.1){}}

\put(6,-12.4){\makebox(0,0)[ct]{\(2;4\)}}
\put(6,-16.4){\makebox(0,0)[ct]{\(1;4\)}}
\put(4,-8){\line(1,-2){2}}
\multiput(8,-14)(0,-4){1}{\line(-1,-1){2}}
\multiput(5.95,-12.05)(0,-4){2}{\framebox(0.1,0.1){}}

\put(7,-12.2){\makebox(0,0)[ct]{\(2;2\)}}
\put(7,-16.2){\makebox(0,0)[ct]{\(1;2\)}}
\put(4,-8){\line(3,-4){3}}
\multiput(8,-14)(0,-4){1}{\line(-1,-2){1}}
\multiput(6.95,-12.05)(0,-4){2}{\framebox(0.1,0.1){}}

\put(8,-14){\line(1,-1){4}}

\put(12.1,-17.8){\makebox(0,0)[lb]{\(2;1\)}}
\multiput(11.95,-18.05)(0,-2){1}{\framebox(0.1,0.1){}}
\put(12,-18){\vector(0,-1){2}}
\put(11.8,-19.9){\makebox(0,0)[rt]{\(\mathcal{T}_\ast^{(5)}(\langle 729,54\rangle-\#2;3-\#1;1-\#2;1-\#1;1-\#2;1)\)}}

\put(9,-18.2){\makebox(0,0)[ct]{\(2;6\)}}
\put(8,-14){\line(1,-4){1}}
\multiput(8.95,-18.05)(0,-4){1}{\framebox(0.1,0.1){}}

\put(10,-18.4){\makebox(0,0)[ct]{\(2;4\)}}
\put(8,-14){\line(1,-2){2}}
\multiput(9.95,-18.05)(0,-4){1}{\framebox(0.1,0.1){}}

\put(11,-18.2){\makebox(0,0)[ct]{\(2;2\)}}
\put(8,-14){\line(3,-4){3}}
\multiput(10.95,-18.05)(0,-4){1}{\framebox(0.1,0.1){}}

\put(-5,-20.9){\makebox(0,0)[cc]{\textbf{TKT:}}}
\put(-3,-21){\makebox(0,0)[cc]{\(\varkappa_1\)}}
\put(-2,-21){\makebox(0,0)[cc]{\(\varkappa_2\)}}
\put(-1,-21){\makebox(0,0)[cc]{\(\varkappa_3\)}}
\put(0,-21){\makebox(0,0)[cc]{\(\varkappa_0\)}}
\put(1,-21){\makebox(0,0)[cc]{\(\varkappa_1\)}}
\put(2,-21){\makebox(0,0)[cc]{\(\varkappa_2\)}}
\put(3,-21){\makebox(0,0)[cc]{\(\varkappa_3\)}}
\put(4,-21){\makebox(0,0)[cc]{\(\varkappa_0\)}}
\put(5,-21){\makebox(0,0)[cc]{\(\varkappa_1\)}}
\put(6,-21){\makebox(0,0)[cc]{\(\varkappa_2\)}}
\put(7,-21){\makebox(0,0)[cc]{\(\varkappa_3\)}}
\put(8,-21){\makebox(0,0)[cc]{\(\varkappa_0\)}}
\put(9,-21){\makebox(0,0)[cc]{\(\varkappa_1\)}}
\put(10,-21){\makebox(0,0)[cc]{\(\varkappa_2\)}}
\put(11,-21){\makebox(0,0)[cc]{\(\varkappa_3\)}}
\put(12,-21){\makebox(0,0)[cc]{\(\varkappa_0\)}}
\put(-5.8,-21.2){\framebox(18.6,0.6){}}

\put(-2,-4.25){\oval(3,1.5)}

\put(2,-6.25){\oval(3,1.5)}

\end{picture}

\end{figure}

}

\noindent
Figure
\ref{fig:TreeTopoGStypeE}
impressively shows that
entering the unnoticed secret door,
which is provided by the bifurcation at the vertex \(\langle 729,54\rangle\),
immediately leads to the long desired \(3\)-tower group \(G\simeq\langle 6561,622\rangle=\langle 729,54\rangle-\#2;4\)
of the imaginary quadratic field \(F=\mathbb{Q}(\sqrt{-34\,867})\).
The siblings topology is emphasized with red color,
and the projection \(G\to\mathfrak{M}\simeq G/G^{\prime\prime}\) is drawn in blue color.



{\tiny

\vspace{0.2in}

\begin{figure}[ht]
\caption{Tree topology of type E in the first excited state}
\label{fig:TreeTopoES1typeE}


\setlength{\unitlength}{0.8cm}
\begin{picture}(18,22)(-6,-21)


\put(-5,0.5){\makebox(0,0)[cb]{Order}}
\put(-5,0){\line(0,-1){18}}
\multiput(-5.1,0)(0,-2){10}{\line(1,0){0.2}}
\put(-5.2,0){\makebox(0,0)[rc]{\(243\)}}
\put(-4.8,0){\makebox(0,0)[lc]{\(3^5\)}}
\put(-5.2,-2){\makebox(0,0)[rc]{\(729\)}}
\put(-4.8,-2){\makebox(0,0)[lc]{\(3^6\)}}
\put(-5.2,-4){\makebox(0,0)[rc]{\(2\,187\)}}
\put(-4.8,-4){\makebox(0,0)[lc]{\(3^7\)}}
\put(-5.2,-6){\makebox(0,0)[rc]{\(6\,561\)}}
\put(-4.8,-6){\makebox(0,0)[lc]{\(3^8\)}}
\put(-5.2,-8){\makebox(0,0)[rc]{\(19\,683\)}}
\put(-4.8,-8){\makebox(0,0)[lc]{\(3^9\)}}
\put(-5.2,-10){\makebox(0,0)[rc]{\(59\,049\)}}
\put(-4.8,-10){\makebox(0,0)[lc]{\(3^{10}\)}}
\put(-5.2,-12){\makebox(0,0)[rc]{\(177\,147\)}}
\put(-4.8,-12){\makebox(0,0)[lc]{\(3^{11}\)}}
\put(-5.2,-14){\makebox(0,0)[rc]{\(531\,441\)}}
\put(-4.8,-14){\makebox(0,0)[lc]{\(3^{12}\)}}
\put(-5.2,-16){\makebox(0,0)[rc]{\(1\,594\,323\)}}
\put(-4.8,-16){\makebox(0,0)[lc]{\(3^{13}\)}}
\put(-5.2,-18){\makebox(0,0)[rc]{\(4\,782\,969\)}}
\put(-4.8,-18){\makebox(0,0)[lc]{\(3^{14}\)}}
\put(-5,-18){\vector(0,-1){2}}


\put(7,0.5){\makebox(0,0)[cc]{Symmetric topology symbol (\(1\)st excited state):}}
{\color{red}
\put(7,-1){\makebox(0,0)[cc]{
\(\overbrace{\mathrm{E}\binom{1}{\rightarrow}}^{\text{Leaf}}\quad
\overbrace{\left\lbrace\mathrm{c}\binom{1}{\rightarrow}\right\rbrace^{2}\ }^{\text{Mainline}}\quad
\overbrace{\mathrm{c}}^{\text{Fork}}\quad
\overbrace{\left\lbrace\binom{2}{\leftarrow}\mathrm{c}\binom{1}{\leftarrow}\mathrm{c}\right\rbrace\ }^{\text{Maintrunk}}\quad
\overbrace{\binom{2}{\leftarrow}\mathrm{E}}^{\text{Leaf}}\)
}}
}
\put(9,-3){\makebox(0,0)[cc]{Transfer kernel types:}}
\put(9,-3.5){\makebox(0,0)[cc]{E.8: \(\varkappa_3=(1231)\), c.21: \(\varkappa_0=(0231)\)}}
\put(9,-4.5){\makebox(0,0)[cc]{Minimal discriminant:}}
\put(9,-5){\makebox(0,0)[cc]{\(-370\,740\)}}

\put(0.1,0.2){\makebox(0,0)[lb]{\(\langle 8\rangle=R\)}}
{\color{red}
\put(0.1,-1.8){\makebox(0,0)[lb]{\(\langle 54\rangle\)}}
\put(1.1,-2.8){\makebox(0,0)[lb]{\(1^{\text{st}}\) bifurcation}}
\put(0.1,-3.8){\makebox(0,0)[lb]{\(1;3\)}}
\put(0.1,-5.8){\makebox(0,0)[lb]{\(1;1\)}}
}
\put(0.1,-7.8){\makebox(0,0)[lb]{\(1;1\)}}
\put(0.1,-9.8){\makebox(0,0)[lb]{\(1;1\)}}
\put(0.1,-11.8){\makebox(0,0)[lb]{\(1;1\)}}
\put(0,0){\circle*{0.2}}
{\color{red}
\multiput(0,-2)(0,-2){3}{\circle*{0.2}}
}
\multiput(0,-8)(0,-2){3}{\circle*{0.2}}
\put(0,0){\line(0,-1){2}}
{\color{red}
\multiput(0,-2)(0,-2){2}{\line(0,-1){2}}
}
\multiput(0,-6)(0,-2){3}{\line(0,-1){2}}
\put(0,-12){\vector(0,-1){2}}
\put(-0.2,-14.2){\makebox(0,0)[rt]{\(\mathcal{T}_\ast^{(2)}(\langle 243,8\rangle)\)}}

\put(-3,-4.2){\makebox(0,0)[ct]{\(1;6\)}}
\put(-3,-8.2){\makebox(0,0)[ct]{\(1;6\)}}
\put(-3,-12.2){\makebox(0,0)[ct]{\(1;6\)}}
\multiput(0,-2)(0,-4){3}{\line(-3,-2){3}}
\multiput(-3,-4)(0,-4){3}{\circle*{0.2}}

\put(-2,-4.4){\makebox(0,0)[ct]{\(1;2\)}}
\put(-2,-8.4){\makebox(0,0)[ct]{\(1;4\)}}
\put(-2,-12.4){\makebox(0,0)[ct]{\(1;4\)}}
\multiput(0,-2)(0,-4){3}{\line(-1,-1){2}}
\multiput(-2,-4)(0,-4){3}{\circle*{0.2}}

\put(-1,-4.2){\makebox(0,0)[ct]{\(1;4\)}}
{\color{red}
\put(-1,-8.2){\makebox(0,0)[ct]{\(1;2\)}}
}
\put(-1,-12.2){\makebox(0,0)[ct]{\(1;2\)}}
\multiput(0,-2)(0,-8){2}{\line(-1,-2){1}}
\multiput(-1,-4)(0,-8){2}{\circle*{0.2}}
{\color{red}
\put(0,-6){\line(-1,-2){1}}
\put(-1,-8){\circle*{0.2}}
}



\multiput(5,-12)(1,0){2}{\vector(-2,1){7.9}}
{\color{cyan}
\multiput(7,-12)(1,0){1}{\vector(-2,1){7.9}}
}

{\color{red}
\put(0,-2){\line(1,-1){4}}
}

{\color{red}
\put(4.1,-5.8){\makebox(0,0)[lb]{\(2;3\)}}
\put(4.1,-7.8){\makebox(0,0)[lb]{\(1;1\)}}
\put(5.1,-8.8){\makebox(0,0)[lb]{\(2^{\text{nd}}\) bifurcation}}
}
\put(4.1,-9.8){\makebox(0,0)[lb]{\(1;1\)}}
\put(4.1,-11.8){\makebox(0,0)[lb]{\(1;1\)}}
\put(4.1,-13.8){\makebox(0,0)[lb]{\(1;1\)}}
{\color{red}
\multiput(3.95,-6.05)(0,-2){2}{\framebox(0.1,0.1){}}
}
\multiput(3.95,-10.05)(0,-2){3}{\framebox(0.1,0.1){}}
{\color{red}
\put(4,-6){\line(0,-1){2}}
}
\multiput(4,-8)(0,-2){3}{\line(0,-1){2}}
\put(4,-14){\vector(0,-1){2}}
\put(3.8,-16.2){\makebox(0,0)[rt]{\(\mathcal{T}_\ast^{(3)}(\langle 729,54\rangle-\#2;3)\)}}

\put(1,-6.2){\makebox(0,0)[ct]{\(2;6\)}}
\put(1,-10.2){\makebox(0,0)[ct]{\(1;6\)}}
\put(1,-14.2){\makebox(0,0)[ct]{\(1;6\)}}
\put(0,-2){\line(1,-4){1}}
\multiput(4,-8)(0,-4){2}{\line(-3,-2){3}}
\multiput(0.95,-6.05)(0,-4){3}{\framebox(0.1,0.1){}}

\put(2,-6.4){\makebox(0,0)[ct]{\(2;2\)}}
\put(2,-10.4){\makebox(0,0)[ct]{\(1;4\)}}
\put(2,-14.4){\makebox(0,0)[ct]{\(1;4\)}}
\put(0,-2){\line(1,-2){2}}
\multiput(4,-8)(0,-4){2}{\line(-1,-1){2}}
\multiput(1.95,-6.05)(0,-4){3}{\framebox(0.1,0.1){}}

\put(3,-6.2){\makebox(0,0)[ct]{\(2;4\)}}
\put(3,-10.2){\makebox(0,0)[ct]{\(1;2\)}}
\put(3,-14.2){\makebox(0,0)[ct]{\(1;2\)}}
\put(0,-2){\line(3,-4){3}}
\multiput(4,-8)(0,-4){2}{\line(-1,-2){1}}
\multiput(2.95,-6.05)(0,-4){3}{\framebox(0.1,0.1){}}

\put(4,-8){\line(1,-1){4}}

\put(8.1,-11.8){\makebox(0,0)[lb]{\(2;1\)}}
\put(8.1,-13.8){\makebox(0,0)[lb]{\(1;1\)}}
\put(9.1,-14.8){\makebox(0,0)[lb]{\(3^{\text{rd}}\) bifurcation}}
\put(8.1,-15.8){\makebox(0,0)[lb]{\(1;1\)}}
\multiput(7.95,-12.05)(0,-2){3}{\framebox(0.1,0.1){}}
\multiput(8,-12)(0,-2){2}{\line(0,-1){2}}
\put(8,-16){\vector(0,-1){2}}
\put(7.8,-18.2){\makebox(0,0)[rt]{\(\mathcal{T}_\ast^{(4)}(\langle 729,54\rangle-\#2;3-\#1;1-\#2;1)\)}}

\put(5,-12.2){\makebox(0,0)[ct]{\(2;6\)}}
\put(5,-16.2){\makebox(0,0)[ct]{\(1;6\)}}
\put(4,-8){\line(1,-4){1}}
\multiput(8,-14)(0,-4){1}{\line(-3,-2){3}}
\multiput(4.95,-12.05)(0,-4){2}{\framebox(0.1,0.1){}}

\put(6,-12.4){\makebox(0,0)[ct]{\(2;4\)}}
\put(6,-16.4){\makebox(0,0)[ct]{\(1;4\)}}
\put(4,-8){\line(1,-2){2}}
\multiput(8,-14)(0,-4){1}{\line(-1,-1){2}}
\multiput(5.95,-12.05)(0,-4){2}{\framebox(0.1,0.1){}}

{\color{red}
\put(7,-12.2){\makebox(0,0)[ct]{\(2;2\)}}
}
\put(7,-16.2){\makebox(0,0)[ct]{\(1;2\)}}
{\color{red}
\put(4,-8){\line(3,-4){3}}
\put(6.95,-12.05){\framebox(0.1,0.1){}}
}
\multiput(8,-14)(0,-4){1}{\line(-1,-2){1}}
\multiput(6.95,-16.05)(0,-4){1}{\framebox(0.1,0.1){}}

\put(8,-14){\line(1,-1){4}}

\put(12.1,-17.8){\makebox(0,0)[lb]{\(2;1\)}}
\multiput(11.95,-18.05)(0,-2){1}{\framebox(0.1,0.1){}}
\put(12,-18){\vector(0,-1){2}}
\put(11.8,-19.9){\makebox(0,0)[rt]{\(\mathcal{T}_\ast^{(5})(\langle 729,54\rangle-\#2;3-\#1;1-\#2;1-\#1;1-\#2;1)\)}}

\put(9,-18.2){\makebox(0,0)[ct]{\(2;6\)}}
\put(8,-14){\line(1,-4){1}}
\multiput(8.95,-18.05)(0,-4){1}{\framebox(0.1,0.1){}}

\put(10,-18.4){\makebox(0,0)[ct]{\(2;4\)}}
\put(8,-14){\line(1,-2){2}}
\multiput(9.95,-18.05)(0,-4){1}{\framebox(0.1,0.1){}}

\put(11,-18.2){\makebox(0,0)[ct]{\(2;2\)}}
\put(8,-14){\line(3,-4){3}}
\multiput(10.95,-18.05)(0,-4){1}{\framebox(0.1,0.1){}}

\put(-5,-20.9){\makebox(0,0)[cc]{\textbf{TKT:}}}
\put(-3,-21){\makebox(0,0)[cc]{\(\varkappa_1\)}}
\put(-2,-21){\makebox(0,0)[cc]{\(\varkappa_2\)}}
\put(-1,-21){\makebox(0,0)[cc]{\(\varkappa_3\)}}
\put(0,-21){\makebox(0,0)[cc]{\(\varkappa_0\)}}
\put(1,-21){\makebox(0,0)[cc]{\(\varkappa_1\)}}
\put(2,-21){\makebox(0,0)[cc]{\(\varkappa_2\)}}
\put(3,-21){\makebox(0,0)[cc]{\(\varkappa_3\)}}
\put(4,-21){\makebox(0,0)[cc]{\(\varkappa_0\)}}
\put(5,-21){\makebox(0,0)[cc]{\(\varkappa_1\)}}
\put(6,-21){\makebox(0,0)[cc]{\(\varkappa_2\)}}
\put(7,-21){\makebox(0,0)[cc]{\(\varkappa_3\)}}
\put(8,-21){\makebox(0,0)[cc]{\(\varkappa_0\)}}
\put(9,-21){\makebox(0,0)[cc]{\(\varkappa_1\)}}
\put(10,-21){\makebox(0,0)[cc]{\(\varkappa_2\)}}
\put(11,-21){\makebox(0,0)[cc]{\(\varkappa_3\)}}
\put(12,-21){\makebox(0,0)[cc]{\(\varkappa_0\)}}
\put(-5.8,-21.2){\framebox(18.6,0.6){}}

\put(-2,-8.25){\oval(3,1.5)}

\put(6,-12.25){\oval(3,1.5)}

\end{picture}

\end{figure}
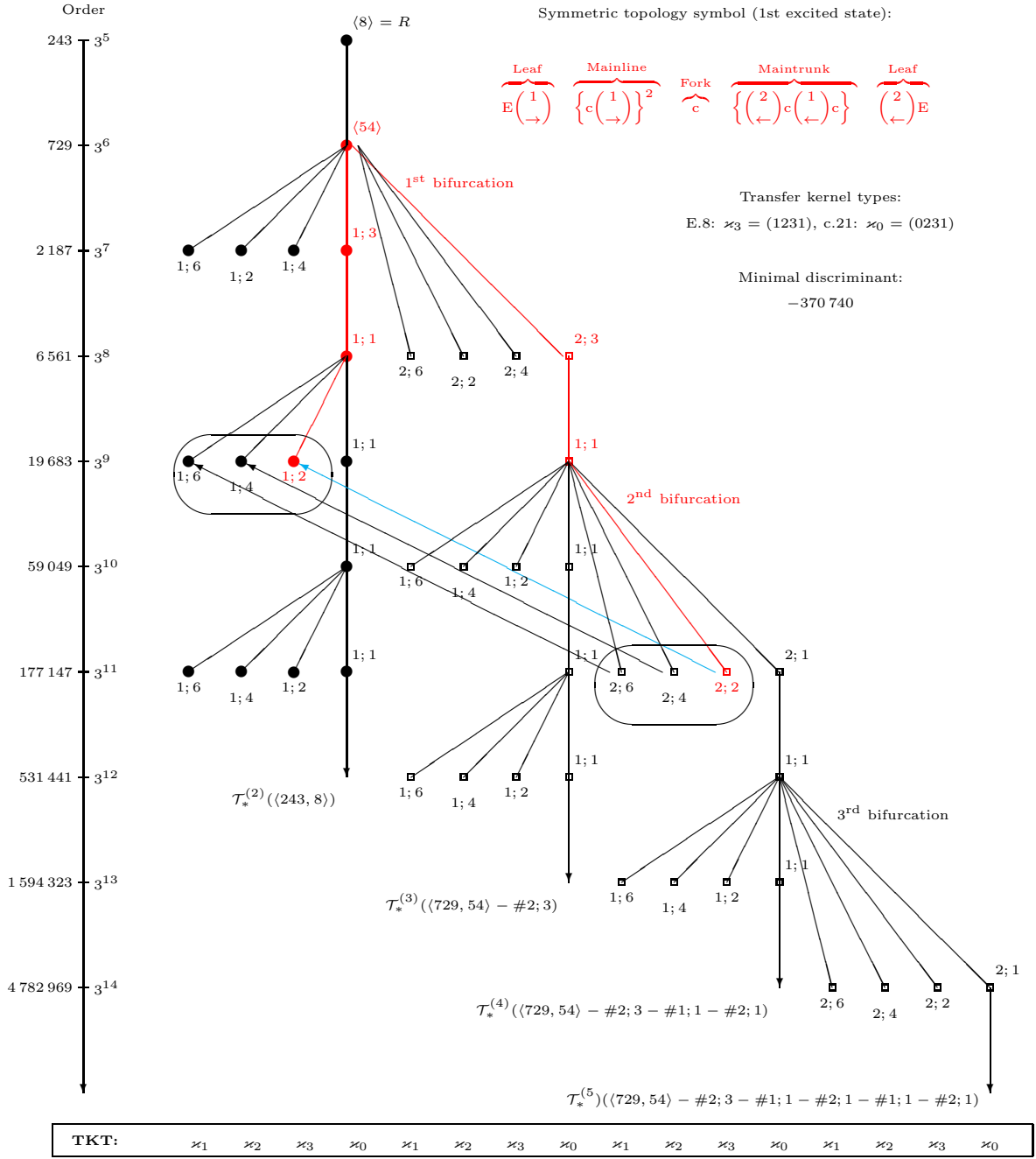

}

\noindent
In Figure
\ref{fig:TreeTopoES1typeE},
we see that the path
to the \(3\)-tower group \(G\simeq\langle 729,54\rangle-\#2;3-\#1;1-\#2;2\)
of the imaginary quadratic field \(F=\mathbb{Q}(\sqrt{-370\,740})\)
contains two bifurcations at \(\langle 729,54\rangle\) and \(\langle 729,54\rangle-\#2;3-\#1;1\).
As before, the fork topology is emphasized with red color,
and the projection \(G\to\mathfrak{M}\simeq G/G^{\prime\prime}\) is drawn in blue color.
Two projection arrows of type \(\mathrm{E}.9\) are black.



{\tiny

\vspace{0.2in}

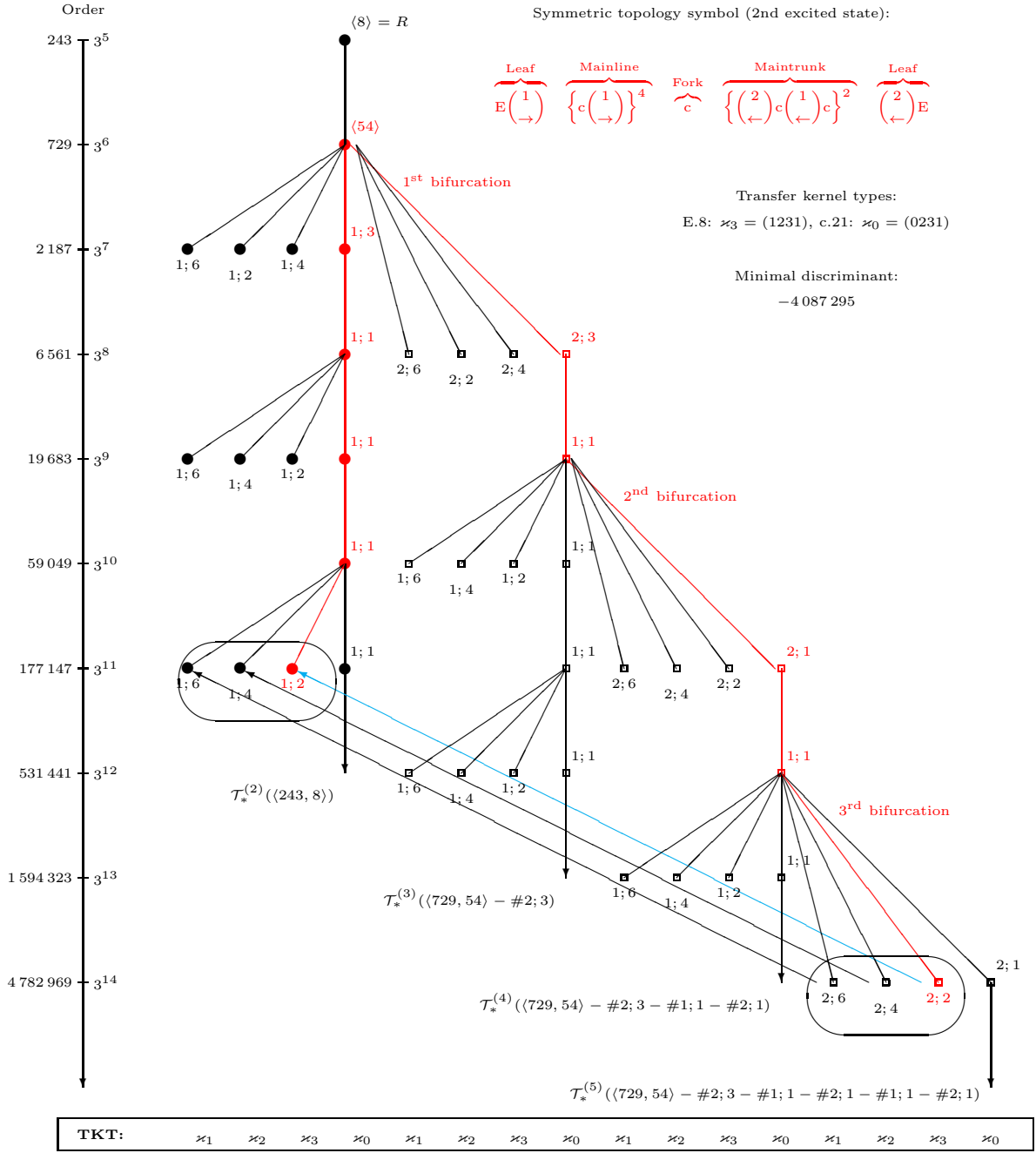
\begin{figure}[ht]
\caption{Tree topology of type E in the second excited state}
\label{fig:TreeTopoES2typeE}


\setlength{\unitlength}{0.8cm}
\begin{picture}(18,22)(-6,-21)


\put(-5,0.5){\makebox(0,0)[cb]{Order}}
\put(-5,0){\line(0,-1){18}}
\multiput(-5.1,0)(0,-2){10}{\line(1,0){0.2}}
\put(-5.2,0){\makebox(0,0)[rc]{\(243\)}}
\put(-4.8,0){\makebox(0,0)[lc]{\(3^5\)}}
\put(-5.2,-2){\makebox(0,0)[rc]{\(729\)}}
\put(-4.8,-2){\makebox(0,0)[lc]{\(3^6\)}}
\put(-5.2,-4){\makebox(0,0)[rc]{\(2\,187\)}}
\put(-4.8,-4){\makebox(0,0)[lc]{\(3^7\)}}
\put(-5.2,-6){\makebox(0,0)[rc]{\(6\,561\)}}
\put(-4.8,-6){\makebox(0,0)[lc]{\(3^8\)}}
\put(-5.2,-8){\makebox(0,0)[rc]{\(19\,683\)}}
\put(-4.8,-8){\makebox(0,0)[lc]{\(3^9\)}}
\put(-5.2,-10){\makebox(0,0)[rc]{\(59\,049\)}}
\put(-4.8,-10){\makebox(0,0)[lc]{\(3^{10}\)}}
\put(-5.2,-12){\makebox(0,0)[rc]{\(177\,147\)}}
\put(-4.8,-12){\makebox(0,0)[lc]{\(3^{11}\)}}
\put(-5.2,-14){\makebox(0,0)[rc]{\(531\,441\)}}
\put(-4.8,-14){\makebox(0,0)[lc]{\(3^{12}\)}}
\put(-5.2,-16){\makebox(0,0)[rc]{\(1\,594\,323\)}}
\put(-4.8,-16){\makebox(0,0)[lc]{\(3^{13}\)}}
\put(-5.2,-18){\makebox(0,0)[rc]{\(4\,782\,969\)}}
\put(-4.8,-18){\makebox(0,0)[lc]{\(3^{14}\)}}
\put(-5,-18){\vector(0,-1){2}}


\put(7,0.5){\makebox(0,0)[cc]{Symmetric topology symbol (\(2\)nd excited state):}}
{\color{red}
\put(7,-1){\makebox(0,0)[cc]{
\(\overbrace{\mathrm{E}\binom{1}{\rightarrow}}^{\text{Leaf}}\quad
\overbrace{\left\lbrace\mathrm{c}\binom{1}{\rightarrow}\right\rbrace^{4}\ }^{\text{Mainline}}\quad
\overbrace{\mathrm{c}}^{\text{Fork}}\quad
\overbrace{\left\lbrace\binom{2}{\leftarrow}\mathrm{c}\binom{1}{\leftarrow}\mathrm{c}\right\rbrace^{2}\ }^{\text{Maintrunk}}\quad
\overbrace{\binom{2}{\leftarrow}\mathrm{E}}^{\text{Leaf}}\)
}}
}
\put(9,-3){\makebox(0,0)[cc]{Transfer kernel types:}}
\put(9,-3.5){\makebox(0,0)[cc]{E.8: \(\varkappa_3=(1231)\), c.21: \(\varkappa_0=(0231)\)}}
\put(9,-4.5){\makebox(0,0)[cc]{Minimal discriminant:}}
\put(9,-5){\makebox(0,0)[cc]{\(-4\,087\,295\)}}

\put(0.1,0.2){\makebox(0,0)[lb]{\(\langle 8\rangle=R\)}}
{\color{red}
\put(0.1,-1.8){\makebox(0,0)[lb]{\(\langle 54\rangle\)}}
\put(1.1,-2.8){\makebox(0,0)[lb]{\(1^{\text{st}}\) bifurcation}}
\put(0.1,-3.8){\makebox(0,0)[lb]{\(1;3\)}}
\put(0.1,-5.8){\makebox(0,0)[lb]{\(1;1\)}}
\put(0.1,-7.8){\makebox(0,0)[lb]{\(1;1\)}}
\put(0.1,-9.8){\makebox(0,0)[lb]{\(1;1\)}}
}
\put(0.1,-11.8){\makebox(0,0)[lb]{\(1;1\)}}
\put(0,0){\circle*{0.2}}
{\color{red}
\multiput(0,-2)(0,-2){5}{\circle*{0.2}}
}
\multiput(0,-12)(0,-2){1}{\circle*{0.2}}
\put(0,0){\line(0,-1){2}}
{\color{red}
\multiput(0,-2)(0,-2){4}{\line(0,-1){2}}
}
\multiput(0,-10)(0,-2){1}{\line(0,-1){2}}
\put(0,-12){\vector(0,-1){2}}
\put(-0.2,-14.2){\makebox(0,0)[rt]{\(\mathcal{T}_\ast^{(2)}(\langle 243,8\rangle)\)}}

\put(-3,-4.2){\makebox(0,0)[ct]{\(1;6\)}}
\put(-3,-8.2){\makebox(0,0)[ct]{\(1;6\)}}
\put(-3,-12.2){\makebox(0,0)[ct]{\(1;6\)}}
\multiput(0,-2)(0,-4){3}{\line(-3,-2){3}}
\multiput(-3,-4)(0,-4){3}{\circle*{0.2}}

\put(-2,-4.4){\makebox(0,0)[ct]{\(1;2\)}}
\put(-2,-8.4){\makebox(0,0)[ct]{\(1;4\)}}
\put(-2,-12.4){\makebox(0,0)[ct]{\(1;4\)}}
\multiput(0,-2)(0,-4){3}{\line(-1,-1){2}}
\multiput(-2,-4)(0,-4){3}{\circle*{0.2}}

\put(-1,-4.2){\makebox(0,0)[ct]{\(1;4\)}}
\put(-1,-8.2){\makebox(0,0)[ct]{\(1;2\)}}
{\color{red}
\put(-1,-12.2){\makebox(0,0)[ct]{\(1;2\)}}
}
\multiput(0,-2)(0,-4){2}{\line(-1,-2){1}}
\multiput(-1,-4)(0,-4){2}{\circle*{0.2}}
{\color{red}
\put(0,-10){\line(-1,-2){1}}
\put(-1,-12){\circle*{0.2}}
}



\multiput(9,-18)(1,0){2}{\vector(-2,1){11.9}}
{\color{cyan}
\multiput(11,-18)(1,0){1}{\vector(-2,1){11.9}}
}

{\color{red}
\put(0,-2){\line(1,-1){4}}
}

{\color{red}
\put(4.1,-5.8){\makebox(0,0)[lb]{\(2;3\)}}
\put(4.1,-7.8){\makebox(0,0)[lb]{\(1;1\)}}
\put(5.1,-8.8){\makebox(0,0)[lb]{\(2^{\text{nd}}\) bifurcation}}
}
\put(4.1,-9.8){\makebox(0,0)[lb]{\(1;1\)}}
\put(4.1,-11.8){\makebox(0,0)[lb]{\(1;1\)}}
\put(4.1,-13.8){\makebox(0,0)[lb]{\(1;1\)}}
{\color{red}
\multiput(3.95,-6.05)(0,-2){2}{\framebox(0.1,0.1){}}
}
\multiput(3.95,-10.05)(0,-2){3}{\framebox(0.1,0.1){}}
{\color{red}
\put(4,-6){\line(0,-1){2}}
}
\multiput(4,-8)(0,-2){3}{\line(0,-1){2}}
\put(4,-14){\vector(0,-1){2}}
\put(3.8,-16.2){\makebox(0,0)[rt]{\(\mathcal{T}_\ast^{(3)}(\langle 729,54\rangle-\#2;3)\)}}

\put(1,-6.2){\makebox(0,0)[ct]{\(2;6\)}}
\put(1,-10.2){\makebox(0,0)[ct]{\(1;6\)}}
\put(1,-14.2){\makebox(0,0)[ct]{\(1;6\)}}
\put(0,-2){\line(1,-4){1}}
\multiput(4,-8)(0,-4){2}{\line(-3,-2){3}}
\multiput(0.95,-6.05)(0,-4){3}{\framebox(0.1,0.1){}}

\put(2,-6.4){\makebox(0,0)[ct]{\(2;2\)}}
\put(2,-10.4){\makebox(0,0)[ct]{\(1;4\)}}
\put(2,-14.4){\makebox(0,0)[ct]{\(1;4\)}}
\put(0,-2){\line(1,-2){2}}
\multiput(4,-8)(0,-4){2}{\line(-1,-1){2}}
\multiput(1.95,-6.05)(0,-4){3}{\framebox(0.1,0.1){}}

\put(3,-6.2){\makebox(0,0)[ct]{\(2;4\)}}
\put(3,-10.2){\makebox(0,0)[ct]{\(1;2\)}}
\put(3,-14.2){\makebox(0,0)[ct]{\(1;2\)}}
\put(0,-2){\line(3,-4){3}}
\multiput(4,-8)(0,-4){2}{\line(-1,-2){1}}
\multiput(2.95,-6.05)(0,-4){3}{\framebox(0.1,0.1){}}

{\color{red}
\put(4,-8){\line(1,-1){4}}
}

{\color{red}
\put(8.1,-11.8){\makebox(0,0)[lb]{\(2;1\)}}
\put(8.1,-13.8){\makebox(0,0)[lb]{\(1;1\)}}
\put(9.1,-14.8){\makebox(0,0)[lb]{\(3^{\text{rd}}\) bifurcation}}
}
\put(8.1,-15.8){\makebox(0,0)[lb]{\(1;1\)}}
{\color{red}
\multiput(7.95,-12.05)(0,-2){2}{\framebox(0.1,0.1){}}
\multiput(8,-12)(0,-2){1}{\line(0,-1){2}}
}
\multiput(7.95,-16.05)(0,-2){1}{\framebox(0.1,0.1){}}
\multiput(8,-14)(0,-2){1}{\line(0,-1){2}}
\put(8,-16){\vector(0,-1){2}}
\put(7.8,-18.2){\makebox(0,0)[rt]{\(\mathcal{T}_\ast^{(4)}(\langle 729,54\rangle-\#2;3-\#1;1-\#2;1)\)}}

\put(5,-12.2){\makebox(0,0)[ct]{\(2;6\)}}
\put(5,-16.2){\makebox(0,0)[ct]{\(1;6\)}}
\put(4,-8){\line(1,-4){1}}
\multiput(8,-14)(0,-4){1}{\line(-3,-2){3}}
\multiput(4.95,-12.05)(0,-4){2}{\framebox(0.1,0.1){}}

\put(6,-12.4){\makebox(0,0)[ct]{\(2;4\)}}
\put(6,-16.4){\makebox(0,0)[ct]{\(1;4\)}}
\put(4,-8){\line(1,-2){2}}
\multiput(8,-14)(0,-4){1}{\line(-1,-1){2}}
\multiput(5.95,-12.05)(0,-4){2}{\framebox(0.1,0.1){}}

\put(7,-12.2){\makebox(0,0)[ct]{\(2;2\)}}
\put(7,-16.2){\makebox(0,0)[ct]{\(1;2\)}}
\put(4,-8){\line(3,-4){3}}
\put(6.95,-12.05){\framebox(0.1,0.1){}}
\multiput(8,-14)(0,-4){1}{\line(-1,-2){1}}
\multiput(6.95,-16.05)(0,-4){1}{\framebox(0.1,0.1){}}

\put(8,-14){\line(1,-1){4}}

\put(12.1,-17.8){\makebox(0,0)[lb]{\(2;1\)}}
\multiput(11.95,-18.05)(0,-2){1}{\framebox(0.1,0.1){}}
\put(12,-18){\vector(0,-1){2}}
\put(11.8,-19.9){\makebox(0,0)[rt]{\(\mathcal{T}_\ast^{(5)}(\langle 729,54\rangle-\#2;3-\#1;1-\#2;1-\#1;1-\#2;1)\)}}

\put(9,-18.2){\makebox(0,0)[ct]{\(2;6\)}}
\put(8,-14){\line(1,-4){1}}
\multiput(8.95,-18.05)(0,-4){1}{\framebox(0.1,0.1){}}

\put(10,-18.4){\makebox(0,0)[ct]{\(2;4\)}}
\put(8,-14){\line(1,-2){2}}
\multiput(9.95,-18.05)(0,-4){1}{\framebox(0.1,0.1){}}

{\color{red}
\put(11,-18.2){\makebox(0,0)[ct]{\(2;2\)}}
\put(8,-14){\line(3,-4){3}}
\multiput(10.95,-18.05)(0,-4){1}{\framebox(0.1,0.1){}}
}

\put(-5,-20.9){\makebox(0,0)[cc]{\textbf{TKT:}}}
\put(-3,-21){\makebox(0,0)[cc]{\(\varkappa_1\)}}
\put(-2,-21){\makebox(0,0)[cc]{\(\varkappa_2\)}}
\put(-1,-21){\makebox(0,0)[cc]{\(\varkappa_3\)}}
\put(0,-21){\makebox(0,0)[cc]{\(\varkappa_0\)}}
\put(1,-21){\makebox(0,0)[cc]{\(\varkappa_1\)}}
\put(2,-21){\makebox(0,0)[cc]{\(\varkappa_2\)}}
\put(3,-21){\makebox(0,0)[cc]{\(\varkappa_3\)}}
\put(4,-21){\makebox(0,0)[cc]{\(\varkappa_0\)}}
\put(5,-21){\makebox(0,0)[cc]{\(\varkappa_1\)}}
\put(6,-21){\makebox(0,0)[cc]{\(\varkappa_2\)}}
\put(7,-21){\makebox(0,0)[cc]{\(\varkappa_3\)}}
\put(8,-21){\makebox(0,0)[cc]{\(\varkappa_0\)}}
\put(9,-21){\makebox(0,0)[cc]{\(\varkappa_1\)}}
\put(10,-21){\makebox(0,0)[cc]{\(\varkappa_2\)}}
\put(11,-21){\makebox(0,0)[cc]{\(\varkappa_3\)}}
\put(12,-21){\makebox(0,0)[cc]{\(\varkappa_0\)}}
\put(-5.8,-21.2){\framebox(18.6,0.6){}}

\put(-2,-12.25){\oval(3,1.5)}

\put(10,-18.25){\oval(3,1.5)}

\end{picture}

\end{figure}

}

\noindent
Figure
\ref{fig:TreeTopoES2typeE}
shows the path
to the \(3\)-tower group \(G\simeq\langle 729,54\rangle-\#2;3-\#1;1-\#2;1-\#1;1-\#2;2\)
of the imaginary quadratic field \(F=\mathbb{Q}(\sqrt{-4\,087\,295})\).
It requires three bifurcations at \(\langle 729,54\rangle\),
\(\langle 729,54\rangle-\#2;3-\#1;1\) and \(\langle 729,54\rangle-\#2;3-\#1;1-\#2;1-\#1;1\).
Again, the fork topology is emphasized with red color,
and the projection \(G\to\mathfrak{M}\simeq G/G^{\prime\prime}\) is drawn in blue color.



\section{Future developments}
\label{s:FutureDevelopments}
\noindent
Fork topologies with significantly higher complexity
and step sizes up to \(3\) and even \(4\)
will be investigated in cooperation with M. F. Newman
\cite{MaNm}
for finite \(3\)-groups with TKT \(\mathrm{F}\).


\section{Acknowledgements}
\label{s:Acknowledgements}
\noindent
The author gratefully acknowledges support from the Austrian Science Fund (FWF):
P 26008-N25.
Indebtedness is expressed to M. F. Newman from the Australian National University, Canberra,
for valuable help in providing evidence of Theorems
\ref{thm:SecondPeriodicity},
\ref{thm:Mainlines},
and
\ref{thm:ExplicitCovers}.





\begin{thebibliography}{XX}
%
\bibitem{As1}
J. A. Ascione,
\textit{On \(3\)-groups of second maximal class},
Ph.D. Thesis,
Austral. National Univ.,
Canberra,
1979.
%
\bibitem{As2}
J. A. Ascione,
\textit{On \(3\)-groups of second maximal class},
Bull. Austral. Math. Soc.
\textbf{21}
(1980),
473--474.
%
\bibitem{BEO2}
H. U. Besche, B. Eick, and E. A. O'Brien,
\textit{The SmallGroups Library --- a Library of Groups of Small Order},
2005,
an accepted and refereed GAP package, available also in MAGMA.
%
\bibitem{BuMa}
M. R. Bush and D. C. Mayer,
\textit{\(3\)-class field towers of exact length \(3\)},
J. Number Theory
\textbf{147}
(2015),
766--777,
DOI 10.1016/j.jnt.2014.08.010.
%
\bibitem{dS}
M. du Sautoy,
\textit{Counting \(p\)-groups and nilpotent groups},
Inst. Hautes \'Etudes Sci. Publ. Math.
\textbf{92}
(2000),
63--112.
%
\bibitem{EkLg}
B. Eick and C. Leedham-Green,
\textit{On the classification of prime-power groups by coclass},
Bull. London Math. Soc.
\textbf{40} (2)
(2008),
274--288.
%
\bibitem{GNO}
G. Gamble, W. Nickel, and E. A. O'Brien,
\textit{ANU p-Quotient --- p-Quotient and p-Group Generation Algorithms},
2006,
an accepted GAP package, available also in MAGMA.
%
\bibitem{MAGMA}
The MAGMA Group,
\textit{MAGMA Computational Algebra System},
Version 2.22-7,
Sydney,
2016, \\
\verb+(http://magma.maths.usyd.edu.au)+.
%
\bibitem{Ma5s}
D. C. Mayer and M. F. Newman,
\textit{Finite \(3\)-groups as viewed from class field theory},
Groups St. Andrews 2013,
Univ. of St. Andrews,
Fife, Scotland, UK,
contributed presentation delivered on 11 August 2013.
\verb+http://www.algebra.at/GroupsStAndrews2013.pdf+.
%
\bibitem{Ma5c}
D. C. Mayer,
\textit{Power-commutator presentations for infinite sequences of \(3\)-groups},
preprint,
2013. \\
\verb+(https://www.researchgate.net/publication/+ \\
\verb+256459683_Power-commutator_presentations_for_infinite_sequences_of_3-groups)+.
%
\bibitem{Ma6}
D. C. Mayer,
\textit{Periodic bifurcations in descendant trees of finite \(p\)-groups},
Adv. Pure Math.
\textbf{5}
(2015),
no. 4,
162--195,
DOI 10.4236/apm.2015.54020,
Special Issue on Group Theory,
March 2015.
%
\bibitem{Ma8}
D. C. Mayer,
\textit{Periodic sequences of \(p\)-class tower groups},
J. Appl. Math. Phys.
\textbf{3}
(2015),
no. 7,
746--756,
DOI 10.4236/jamp.2015.37090,
First International Conference on Groups and Algebras, 2015,
Shanghai.
%
\bibitem{Ma9}
D. C. Mayer,
\textit{Artin transfer patterns on descendant trees of finite \(p\)-groups},
Adv. Pure Math.
\textbf{6}
(2016),
no. 2,
66 -- 104,
DOI 10.4236/apm.2016.62008,
Special Issue on Group Theory Research,
January 2016.
%
\bibitem{Ma10}
D. C. Mayer,
\textit{New number fields with known \(p\)-class tower},
Tatra Mt. Math. Pub.,
\textbf{64}
(2015),
21--57,
DOI 10.1515/tmmp-2015-0040,
Special Issue on Number Theory and Cryptology \lq 15.
%
\bibitem{Ma12}
D. C. Mayer,
\textit{Annihilator ideals of two-generated metabelian \(p\)-groups}.
(arXiv: 1603.09288v1 [math.GR] 30 Mar 2016.)
%
\bibitem{Ma15}
D. C. Mayer,
\textit{Recent progress in determining \(p\)-class field towers},
Gulf J. Math. (Dubai, UAE)
\textbf{4}
(2016),
no. 4,
74--102,
ISSN 2309-4966.
%
\bibitem{Ma15b}
D. C. Mayer,
\textit{Recent progress in determining \(p\)-class field towers},
1st International Colloquium of
Algebra, Number Theory, Cryptography and Information Security \(2016\),
Facult\'e Polydisciplinaire de Taza, Universit\'e Sidi Mohamed Ben Abdellah,
F\`es, Morocco,
invited keynote delivered on 12 November 2016, \\
\verb+http://www.algebra.at/ANCI2016DCM.pdf+.
%
\bibitem{MaNm}
D. C. Mayer and M. F. Newman,
\textit{Finite \(3\)-groups with transfer kernel type \(\mathrm{F}\)},
in preparation, \\
\verb+http://www.algebra.at/TransferSectionF.pdf+.
%
\bibitem{Ne}
B. Nebelung,
\textit{Klassifikation metabelscher \(3\)-Gruppen
mit Faktorkommutatorgruppe vom Typ \((3,3)\)
und Anwendung auf das Kapitulationsproblem},
Inauguraldissertation,
Universit\"at zu K\"oln,
1989.
%
\bibitem{Nm}
M. F. Newman,
\textit{Determination of groups of prime-power order},
pp. 73--84,
in: Group Theory, Canberra, 1975,
Lecture Notes in Math.,
vol. \textbf{573},
Springer,
Berlin,
1977.
%
\bibitem{OB}
E. A. O'Brien,
\textit{The \(p\)-group generation algorithm},
J. Symbolic Comput.
\textbf{9}
(1990),
677--698.
%
\bibitem{SoTa}
A. Scholz und O. Taussky,
\textit{Die Hauptideale der kubischen Klassenk\"orper imagin\"ar quadratischer Zahlk\"orper:
ihre rechnerische Bestimmung und ihr Einflu\ss\ auf den Klassenk\"orperturm},
J. Reine Angew. Math.
\textbf{171}
(1934),
19--41.
%
\bibitem{Sh}
I. R. Shafarevich,
\textit{Extensions with prescribed ramification points} (Russian),
Publ. Math., Inst. Hautes \'Etudes Sci. 
\textbf{18}
(1964),
71--95.
(English transl. by J. W. S. Cassels in
Amer. Math. Soc. Transl.,
II. Ser.,
\textbf{59}
(1966),
128--149.)
%
\end{thebibliography}
\end{document}